\documentclass[a4paper,11pt]{amsart}
\usepackage{amssymb, amsmath, mathrsfs}
\usepackage{mathtools}
\mathtoolsset{showonlyrefs}
\usepackage[inline, final]{showlabels}

\usepackage{tikz-cd}
\usepackage{mathrsfs}
\usepackage{enumitem}
\usepackage{shuffle}
\usepackage{pdfpages}

\usepackage{subcaption}

\usepackage{float}

\usepackage[]{todonotes}
\usepackage[
 pdfauthor={Yiannis N. Petridis and Dimitrios Lekkas},
 pdftitle={A relative trace formula and counting geodesic arcs in the hyperbolic plane},
 pdfsubject={Primary XXXX; Secondary  YYYY},
 pdfkeywords={},
 pdfproducer={LaTeX2e},
 pdfcreator={Pdflatex},
 pdfdisplaydoctitle =true,
 psdextra, hidelinks]
 {hyperref}
\usepackage{pdfsync}

\newtheorem{theorem}{Theorem}[section]
\newtheorem{proposition}[theorem]{Proposition}
\newtheorem{lem}[theorem]{Lemma}

\newtheorem{conjecture}[theorem]{Conjecture}
\newtheorem*{rem*}{Remark}

\theoremstyle{definition}
\newtheorem{defn}[theorem]{Definition}

\newtheorem{remark}{Remark}
\DeclareMathOperator{\Tr}{Tr}

\title[Counting geodesic arcs]{A relative trace formula and counting geodesic arcs in the hyperbolic plane}
\date{today}
\author[D. Lekkas]{Dimitrios Lekkas}
\address{Faculty of Computing, Mathematics, Engineering \& Natural Sciences, Northeastern University London, Devon House, 58 St Katharine's Way, London E1W 1LP, United Kingdom}
\email{dimitrios.lekkas@nulondon.ac.uk}

\author[Y. Petridis]{Yiannis N. Petridis}
\address{Department of Mathematics, University College London, Gower Street, London WC1E 6BT, United Kingdom}
\email{i.petridis@ucl.ac.uk}

\thanks{}

\date{\today}

\usepackage{xcolor}

\usepackage{tikz}
\usetikzlibrary{arrows.meta}
\usetikzlibrary{decorations.markings}

\usepackage{tikz-3dplot}

\usepackage{pgfplots}

\definecolor{plum}{rgb}{0.36078, 0.20784, 0.4}
\definecolor{chameleon}{rgb}{0.30588, 0.60392, 0.023529}
\definecolor{cornflower}{rgb}{0.12549, 0.29020, 0.52941}
\definecolor{scarlet}{rgb}{0.8, 0, 0}
\definecolor{brick}{rgb}{0.64314, 0, 0}
\definecolor{sunrise}{rgb}{0.80784, 0.36078, 0}
\definecolor{lightblue}{rgb}{0.15,0.35,0.75}

\pgfplotsset{compat=1.17}

\usepackage[backend=biber,style=alphabetic]{biblatex}
\begin{document}

\begin{abstract}
We study a modification of the hyperbolic circle problem: instead of all elements of a Fuchsian group $\Gamma$,  we consider the double cosets by two hyperbolic subgroups. This has a geometric interpretation in terms of the number of common perpendiculars between two closed geodesics for $\Gamma \backslash\mathbb{H}$. We prove an explicit relative trace formula, which is flexible for the counting problem. Using a large sieve inequality developed by the first author and Voskou, we prove a new bound in mean square for the error term of order $O(X^{1/2}\log X)$. We conjecture that this is the correct order of growth. Along the way we provide a new proof of the pointwise error bound $O(X^{2/3})$, originally proved by Good.
\end{abstract}

\maketitle
\section{Counting cosets on Fuchsian groups}\label{Counting cosets on Fuchsian groups}

For the hyperbolic plane $ \mathbb H$ and two points $z$ and $w$ in it, we denote  $\rho (z, w)$
the hyperbolic distance between them.  Let $
\Gamma$ be a cocompact 
subgroup of $\hbox{PSL}(2, \mathbb R)$. The classical hyperbolic lattice point problem asks to estimate the number
of points in the orbit $\Gamma z$ within a disk of radius $\cosh^{-1}(X/2)$  centered  at $w$, i.e. to give
an asymptotic formula for $$N(X, z, w )=\#\{\gamma \in \Gamma : 2\cosh \rho (\gamma z, w) \le X\}.$$
Huber \cite{Hu59},  Selberg \cite{Selberg}, Good \cite{Good} proved that
$$N(X, z, w)=M (X, z, w)+O(X^{2/3}),$$ where
$$M(X, z, w)\sim \frac{2\pi}{\hbox{vol}(\Gamma\backslash \mathbb H)} X, \quad X\to \infty.$$
The main term contains contributions from the eigenfunctions of the Laplace operator with small eigenvalues. 

Huber \cite{HuHel}, and, later, Good \cite{Good} considered other orbit-counting problems in hyperbolic space, stemming from the classification of 
 elements $\gamma \in \hbox{PSL}(2,\mathbb{R})$
into elliptic, parabolic and hyperbolic elements. Such a $\gamma$ fixes a point, cusp or geodesic respectively.  For $i=1,2$, let $H_i$ be a subgroup of $\Gamma$, which is the stabilizer of a point, cusp or geodesic. Considering the double cosets in $H_1 \backslash \Gamma /H_2$, instead of the full group $\Gamma$, leads to nine different counting problems. Good \cite{Good} studied all these  problems in a unified way. In an appropriate parameter $X$, the error term in the counting is $O\big(X^{2/3}\big)$. The error term has not been improved for any of the nine problems. Unfortunately, Good's book is hard to understand because of complicated notation. 
 Huber studied  the case $H_1$ hyperbolic and $H_2$ elliptic  in \cite{Hu98} as well.
 
 Our goal here is to use a different, more flexible technique for $H_1$ and $H_2$ hyperbolic.  We modify Huber's method in \cite{Hu98} to prove a relative trace formula that allows fairly general test functions. Using the extra flexibility we study the error in mean square.
 
 More specifically, this problem concerns the estimation of
$$\widetilde{N}(X,l_1,l_2)= \# \Big\{\gamma \in  H_1\backslash \Gamma /H_2 | \inf_{\substack{z\in l_1 \\ w\in l_2}}\cosh({\rho} (\gamma z,w))\le X\Big\} \: ,$$
where $H_1\backslash\Gamma/H_2$ is the double coset decomposition of $\Gamma$ by hyperbolic subgroups $H_1 , H_2$, corresponding to closed geodesics $l_1$ and $l_2$. By conjugating the group $\Gamma$, we can assume that $l_1$ lies on the imaginary axis $I$. Moreover, we take $H_1=H_2$. The distance between $l_1$ and $\gamma l_1$ is the length of the geodesic segment that is orthogonal to both geodesics, see Fig. \ref{fig:my_label}.
\begin{figure}
    \centering

    \begin{tikzpicture}[scale=1.25,domain=0:13]
	\tikzstyle{axisarrow} = [-{Latex[inset=0pt,length=7pt]}]

	\clip(-2,-1.25) rectangle (6,4);

	\draw [cornflower!30,step=0.2,thin] (-2,0) grid (6,4);
	\draw [cornflower!60,step=1.0,thin] (-2,0) grid (6,4);

	\draw[thick,axisarrow] (-1,-1) -- (-1,4);
	\node[inner sep=0pt] at (-1.3,3.7) {$I$};
	\node[inner sep=0pt] at (5.75,+0.33) {$\Re z$};

\draw[teal] (-1,2.15) -- (-1,2.35) -- (-0.8,2.35) -- (-0.8,2.15) -- (-1,2.15);
\draw[blue] (-1,2.77) -- (-1,2.97) -- (-0.8,2.97) -- (-0.8,2.77) -- (-1,2.77);

\draw[teal] (0.19,1.18) -- (0.36,1.22) -- (0.385,1);
\draw[blue] (2.07,1.25) -- (2.21,1.415) -- (2.35,1.262);
	\node at (-1.2,2.45) [] {$l$};
	\node at (1,1.2) [] {$\gamma_1 I$};
	\node at (5,2.1) [] {$\gamma_2 I$};
	\draw[brick,ultra thick] (-1,1) -- (-1,3);	
	\begin{scope}[shift={(4,0)}]
	\def\Radius{2.1}
  \path
    (-\Radius, 0) coordinate (A)
    -- coordinate (M)
    (\Radius, 0) coordinate (B)
    (M) +(60:\Radius) coordinate (C)
    +(120:\Radius) coordinate (D)
  ;
  \draw[blue]
    (B) arc(0:180:\Radius) -- cycle
; \end{scope}
\begin{scope}[shift={(0.5,0)}]
	\def\Radius{1}
  \path
    (-\Radius, 0) coordinate (A)
    -- coordinate (M)
    (\Radius, 0) coordinate (B)
    (M) +(60:\Radius) coordinate (C)
    +(120:\Radius) coordinate (D)
  ;
  \draw[teal]
    (B) arc(0:180:\Radius) -- cycle
;    
\end{scope}

\draw[teal] (0.2,0.95) arc (0:90:1.2);
\draw[blue] (2.2,1.1) arc (35:90:3.9);
	\draw[thick,axisarrow] (-2,0) -- (6,0);
\end{tikzpicture}	
\caption{The geometric interpretation for the counting problem of the double cosets.}
\label{fig:my_label}
\end{figure}
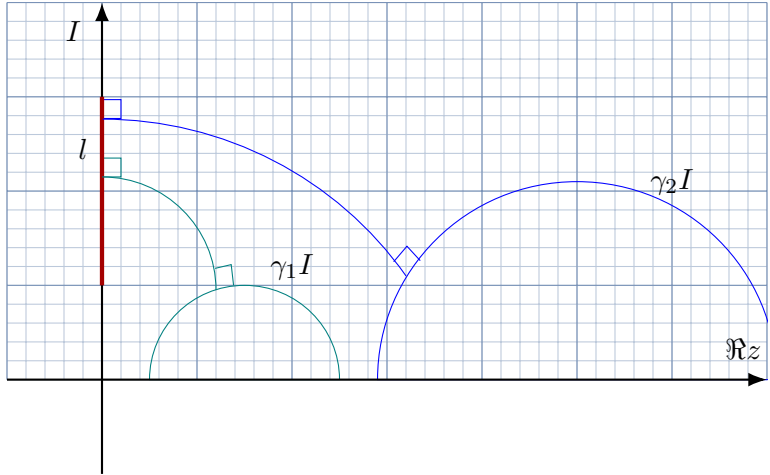
Martin, McKee and Wambach in \cite[Lemma 1]{MMW} relate this distance to $\delta(\gamma):=2|ad+bc|$, for a given $\gamma = \begin{bmatrix}
a & b \\
c & d
\end{bmatrix}  
 \in H_1\backslash\Gamma/H_1$:
\begin{lem}[\cite{MMW}]\label{udist}
For $\gamma \in \rm{PSL}(2,\mathbb{R})$ such that $abcd\neq 0$ and $u(z,w)$ the point-pair invariant defined in \eqref{uzw} we have  
\[ \inf(u(\gamma \cdot ix,iy) \; | x,y \in \mathbb{R}^{+})= 
\begin{cases}
\delta(\gamma)-2, & abcd>0 \; , \\
0, & abcd<0 \; .
\end{cases}
\]
Thus
$$ \max(\delta(\gamma),2)= 2\cosh({\rm dist}(\gamma I,I))\: ,$$
where ${\rm dist}(\gamma I,I)=\inf_{z,w\in I}\rho(\gamma z,w)$.
\end{lem}
Similarly to \cite[p.11]{MMW} and \cite[Lem. 8]{Tsuz} we may assume that $\gamma \in \Gamma - H_1$ satisfies the condition $abcd\neq 0$, which implies $\delta(\gamma)\neq 2$. Such elements are called regular.
 There are finitely many double coset representatives $\gamma$ in $H_1\backslash \Gamma /H_1$ for which $\delta(\gamma)<2$ or equivalently $abcd<0$, see \cite[p. 11--12]{MMW}. Those elements are called exceptional, see Fig. \ref{figexc}.
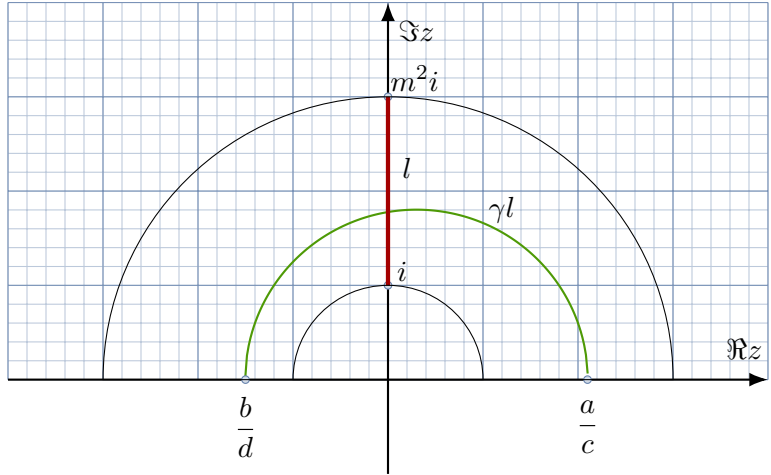
\begin{figure}
\centering
\begin{tikzpicture}[scale=1.25,domain=0:13]
	\tikzstyle{axisarrow} = [-{Latex[inset=0pt,length=7pt]}]

	\clip(-2,-1.25) rectangle (6,4);

	\draw [cornflower!30,step=0.2,thin] (-2,0) grid (6,4);
	\draw [cornflower!60,step=1.0,thin] (-2,0) grid (6,4);

	\draw[thick,axisarrow] (2,-1) -- (2,4);
	\node[inner sep=0pt] at (2.3,3.7) {$\Im z$};
	\draw[thick,axisarrow] (-2,0) -- (6,0);
	\node[inner sep=0pt] at (5.75,+0.33) {$\Re z$};

	\draw [chameleon,thick,domain=1.6973:5.303,samples=200] plot ({\x-1.2}, {sqrt(-(\x-3.5)^2 + 1 + (2-3.5)^2 )});
	
	\node at (0.5,0) [circle,draw=cornflower!70,fill=cornflower!20,inner sep=1pt] {};
	\node at (4.1,0) [circle,draw=cornflower!70,fill=cornflower!20,inner sep=1pt] {};
	\node at (2,1) [circle,draw=cornflower!70,fill=cornflower!20,inner sep=1pt] {};
	\node at (2,3) [circle,draw=cornflower!70,fill=cornflower!20,inner sep=1pt] {};

	\node at (2.2,2.25) [] {$l$};
	\node at (3.2,1.8) [] {$\gamma l$};
	\draw[brick,ultra thick] (2,1) -- (2,3);

	\node at (0.5,-0.5) [] {$\displaystyle \frac{b}{d}$};
	\node at (4.1,-0.5) [] {$\displaystyle \frac{a}{c}$};
	\node at (2.15,1.15) [] {$\displaystyle i$};
	\node at (2.28,3.2) [] {$\displaystyle m^2i$};
	\begin{scope}[shift={(2,0)}]

	\def\Radius{3}
  \path
    (-\Radius, 0) coordinate (A)
    -- coordinate (M)
    (\Radius, 0) coordinate (B)
    (M) +(60:\Radius) coordinate (C)
    +(120:\Radius) coordinate (D)
  ;
  \draw
    (B) arc(0:180:\Radius) -- cycle
;    
	\def\Radius{1}
  \path
    (-\Radius, 0) coordinate (A)
    -- coordinate (M)
    (\Radius, 0) coordinate (B)
    (M) +(60:\Radius) coordinate (C)
    +(120:\Radius) coordinate (D)
  ;
  \draw
    (B) arc(0:180:\Radius) -- cycle
;    
\end{scope}

\end{tikzpicture}	
\caption{A closed geodesic $l$ with norm $m^2$ and its image by an exceptional element $\gamma$.}
    \label{figexc}

\end{figure}

The condition $\inf(\cosh({ \rho} (\gamma z,w)| {{z, w\in l_1}}))\le X$ is equivalent to $$B(\gamma):= |ad+bc| =\delta(\gamma)/2<X, \quad \text{for} \quad \delta(\gamma)>2.$$ 

Set
$$N(X,l_1)= \# \Big\{\gamma \in  H_1\backslash \Gamma /H_1 | B(\gamma)<X \Big\}.$$
Since the exceptional double cosets are finite, we have that $$N(X,l_1)=\widetilde{N}(X,l_1,l_1)+O(1).$$

Therefore, our problem is reduced to the estimation of $N(X,l_1)$.\\
\begin{remark} For the case of two different hyperbolic subgroups $H_1$ and $H_2$, 
we have $$N(X,l_1,l_2)= \# \Big\{\gamma \in  H_1\backslash \Gamma /H_2 | B(\tau^{-1} \gamma)<X \Big\},$$
where $\tau \in \hbox{PSL}_2(\mathbb{R})$ is such that $\tau^{-1} \cdot l_2$ lies on the imaginary axis. See 
\cite[Ch. 6]{MMW} and Remark \ref{Rem}.
    \end{remark}

\subsection{Statements of our results}
Suppose that $\Gamma$ is cocompact and torsion-free. Let $\{u_j\}_{j=0}^{\infty}$ be a complete orthonormal system of real-valued eigenfunctions for the discrete spectrum of the hyperbolic Laplacian, with corresponding eigenvalues $\lambda_j=s_j(1-s_j)$. We also set $s_j=1/2+it_j$. We call the eigenfunctions $u_j$ Maass forms. The eigenvalues $\lambda_j$ with $\lambda_j<1/4$, which is equivalent to $ 1/2<s_j\le 1$, are called small eigenvalues.  

We provide a new proof of Good's theorem \cite[Th. 4]{Good} with explicit main term for $H_1 \backslash \Gamma / H_1$ with $H_1$ a hyperbolic subgroup corresponding to $l:=l_1$.
\begin{theorem}\label{mainthm}
We have
$$N(X,l)= \sum_{1/2<s_j\le1}\frac{2}{\pi}\gamma_1(s_j)|\hat{u}_j|^2X^{s_j}+O\big(X^{2/3}\big),$$
where $$\gamma_1(s)=\frac{\pi}{2}\frac{\Gamma(s+1/2)\Gamma((s+1)/2)}{\Gamma(s/2)^3}\frac{4}{s(2s-1)} \; ,$$
and $$\hat{u}_j=\int_lu_jds.$$ The $O$-estimate depends on the geodesic $l$.
\end{theorem}
\begin{remark}
    Tsuzuki \cite{Tsuz} studied the problem when $l_1$ and $l_2$ are the same, with an error term $O\big(X^{5/6}\big)$, which is worse than the general result of Good: $O(X^{2/3})$, \cite[Th. 4]{Good}.
 The gamma factors $\gamma_1(s)$ match with \cite[Th. 25]{Tsuz}. The asymptotics in Theorem \ref{mainthm} agree with \cite[Th. 25]{Tsuz} up to a  a factor of $2$, see also discussion 2.1 in \cite{Tsuz}. 
\end{remark}
\begin{remark}
    Good \cite{Good} studied four different subcases depending on the sign of $ab$ and $ac$. In this work we only study the aggregate counting for this problem. For further work on the refined counting problem and its geometric significance we refer to \cite{Voskou}.
\end{remark}

We call $\hat{u}_j$ the period of the real-valued Maass form $u_j$ along the geodesic $l$. 
We define the main term of our hyperbolic lattice counting problem as
$$M(X,l):= \sum_{1/2< s_j \le 1} \frac{2}{\pi}\gamma_1(s_j)|\hat{u}_j|^2X^{s_j}, $$ and the error term as 
$$ E(X,l)=N(X,l)-M(X,l).$$
Similarly to \cite[Th. 1.2]{CP}, we want to show that, on average, $E(X,l)$ is $O\big(X^{1/2}\log X\big)$. For this purpose, the first author and Voskou proved the following large sieve inequality for the periods $\hat{u}_j$:
\begin{theorem}[\cite{Lsieve}]\label{Sieve}
Let $T,X>1$ and $x_1, \dots , x_R \in [X,2X]$. Assume that $|x_\nu-x_\mu| > \delta >0$ for $\nu \neq \mu$. Then
\begin{align}
    \begin{split}
        \sum_{\nu=1}^R \Big| \sum_{|t_j|\le T}a_jx_\nu^{it_j}\hat{u}_j\Big|^2 \ll \big(T+X\delta^{-1}\big)||\mathbf{a}||_*^2,
    \end{split}
\end{align}
where $$||\mathbf{a}||_*= \Big(\sum_{|t_j| \le T}|a_j|^2\Big)^{1/2}.$$
\end{theorem}
Note that this is analogous to Chamizo's large sieve inequality \cite[Theorem 2.2]{Cham1}, who worked with values of Maass forms $u_j(z)$ instead of the periods $\hat{u}_j$. \par
We use Theorem \ref{Sieve} to prove the following result about the mean square of the error term $E(X,l)$:
\begin{proposition}\label{SecondMomProp}
Let $X>2$ and $\delta>0$. Let $X_1,X_2, \dots ,X_R \in [X,2X]$ such that $|X_i-X_j|>\delta$, when $i \ne j$. Then the following bound holds:
$$\sum_{m=1}^R|E(X_m,l)|^2 \ll R^{1/3}X^{4/3}\log X+\delta^{-1}X^2\log^2X.$$
\end{proposition}
Using Proposition \ref{SecondMomProp}, we prove the following bounds for the second moment of the error term $E(X,l)$:
\begin{theorem}\label{SecondMomthm}
If $R\delta \gg X$ and $R>X^{1/2}$, then
\begin{equation}\label{1/R}
    \frac{1}{R}\sum_{m=1}^R|E(X_m,l)|^2 \ll X\log^2X. 
\end{equation}
Letting $R \to \infty$, we get
\begin{equation}\label{1/X}
    \frac{1}{X}\int_X^{2X}|E(x,l)|^2dx \ll X\log^2X.
\end{equation} 
\end{theorem}
From Theorems \ref{mainthm} and \ref{SecondMomthm} we can formulate the analogous conjecture to \cite[Conj. 5.7]{CP}:
\begin{conjecture}
For every $\epsilon>0$ the error term $E(X,l)$ satisfies the bound
$$E(X,l)=O\big(X^{1/2+\epsilon}\big),$$
where the estimate depends on $l$ and $\epsilon$.
\end{conjecture}

\subsection{Structure of the paper}
In section \ref{Preliminaries} we explain Huber's framework and some of his main results on the counting problem corresponding to orbits by double cosets $H_1\backslash \Gamma /H_2$, where $H_2$ is an elliptic subgroup of $\Gamma$.
In section \ref{A relative trace formula} we develop a new relative trace formula (Theorem \ref{traceformula}) that is suitable for our counting problem.
In section \ref{The geometric side of the trace formula} we choose suitable test functions in the relative trace formula to relate to the counting $N(X,l)$.
In section \ref{The spectral side of the relative trace formula} we analyze the spectral side of the relative trace formula for our choice of test functions. In section \ref{Proof of counting} we prove Theorem \ref{mainthm} and in section \ref{meansquare} we prove Theorem \ref{SecondMomthm}. 

\section{Preliminaries}\label{Preliminaries}
Our main reference for this section is \cite{IW}. 
Let $\mathbb{H}$ be the hyperbolic upper half-plane 
equipped with the hyperbolic metric 
$$ds^2= \frac{dx^2+dy^2}{y^2},$$
and  measure element  
$d\mu(z)={dxdy}/{y^2}.$
For two points $z,w\in \mathbb{H}$ we define the point-pair invariant $u$ by
\begin{equation}\label{uzw}
    u(z,w)= \frac{|z-w|^2}{4\Im z\Im w}.
\end{equation}
The hyperbolic distance $\rho$ is given by
\begin{equation}\label{rho}
\cosh \rho(z,w) = 1+2u(z,w).
\end{equation}
We study functions $f : \mathbb{H} \rightarrow \mathbb{C}$ periodic under the action of $\Gamma$, i.e. 
\begin{equation}\label{eqn:ac}
f(\gamma z)= f(z),\quad \gamma \in \Gamma ,\quad z\in \mathbb{H}.
\end{equation}
We call these functions automorphic. 

The  inner product of two functions $f,g \in L^2(\Gamma \backslash \mathbb{H})$ is defined by
$$\langle f,g \rangle = \int_{\Gamma \backslash \mathbb{H}} f(z)\overline{g(z)}d\mu(z),$$
and the hyperbolic Laplace operator is 
$$ \Delta = -y^2 \Bigg( \frac{\partial^2}{\partial x^2} + \frac{\partial^2}{\partial y^2} \Bigg).$$
A smooth function $f \in L^2(\Gamma \backslash \mathbb{H}$) that satisfies the automorphic condition \eqref{eqn:ac} and is also an eigenfunction of the Laplace operator is called a Maass form. We denote these by $u_j$ with $\lambda_j$ the corresponding eigenvalues.

An element $g$ of ${\rm PSL}(2,\mathbb{R})$ is called hyperbolic if $|\Tr(g)|>2$.
Huber studied the lattice counting problem in conjugacy classes. Let $\mathfrak{T}$ be a hyperbolic conjugacy class, and
let $\mu$ be the length of the corresponding closed geodesic.
The lattice counting problem in conjugacy classes counts
\begin{equation}\label{NT}
N(t,\mathfrak{T},z)=\# \Big\{\gamma \in \mathfrak{T} \: | \: \rho(\gamma z ,z) \leq t \Big\}.
\end{equation}
Huber in \cite[Satz B]{HuHel} proved that
$$N(t,\mathfrak{T},z) \sim c\frac{\mu}{\sinh\big(\frac{\mu}{2}\big)}e^{t/2},$$
as $t \to \infty$.
We notice that the asymptotic growth is exponential in $t$, since Huber counts the number of $\gamma \in \mathfrak{T}$ such that $\rho(\gamma z ,z) \leq t$. With a parametrization $X=\sinh(t/2)/\sinh(\mu/2)$, the growth is polynomial.\\
We now explain Huber's framework in \cite{Hu98}. We introduce a new system of coordinates $(u,v)$ on the upper half-plane: 
let $z=x+iy$ and define \begin{equation}\label{uv}
u(z)=\log|z| \quad  \text{and} \quad  v(z)=-\arctan\Big(\frac{x}{y}\Big)  .\end{equation} 
The range of  these variables is $$-\infty<u(z)<+\infty \quad \text{and} \quad -\frac{\pi}{2}<v(z)<\frac{\pi}{2} . $$
\\
From the definitions of $u$ and $v$ it follows that $$\cos v(z)=\frac{y}{|z|} \quad \text{and} \quad \sin v(z)= -\frac{x}{|z|}.$$
Moreover, for a diagonal hyperbolic element $P= \begin{bmatrix}
\lambda & 0 \\
0 & \lambda^{-1}
\end{bmatrix}$ we have (see equations \cite[p. 19, Eq.23]{Hu98}) 
$$ u(Pz)=u(z)+2\log\lambda \quad \text{and} \quad v(Pz)=v(z),$$
where $\lambda^2$ is the norm of $P$. Hence, the action of diagonal hyperbolic elements does not change the coordinate $v$.  
The hyperbolic metric now becomes $ds^2=\big(du(z)^2+dv(z)^2\big)\big/\cos^2v(z)$.

Huber \cite[p. 16, Eq. 4]{Hu98} studied the spectral expansion of the series
$$A(f)(z)=\sum_{\gamma \in \mathfrak{T}}f\Bigg(\frac{\cosh(z,\gamma z)-1}{\cosh \mu -1}\Bigg),$$
for $f\in C_0^{*}[1,\infty)$: the space of real functions of
compact support that are bounded in $[1,\infty)$ and have at most finitely many discontinuities. Since $\Gamma$ is cocompact and $f$ has compact support, $A(f)(z)$ is finite.
Using \cite[p.~20, Eq. 26]{Hu98}, Huber rewrote this series in the coordinates $u,v$. For $\mathfrak{T}$ the conjugacy class of $P$, the result is
\begin{equation}\label{Afz}
A(f)(z)=\sum_{\gamma \in \langle P \rangle \backslash \Gamma}f\Big(\frac{1}{\cos ^2v(\gamma z)}\Big).
\end{equation}
The function $A(f)(z)$ is automorphic and with appropriate choice of $f$ can be used to study $N(t,\mathfrak{T},z)$ via its spectral expansion.
\begin{theorem}\label{hub}\rm \cite[17]{Hu98}
The automophic function $A(f)(z)$ has the following spectral expansion:
$$A(f)(z)= \sum_{j}c_j(f)u_j(z),$$
where $c_j(f)=2\hat{u}_jd_{t_j}(f)$. The transform $d_t(\cdot)$, called Huber transform, is given by
\begin{equation}\label{df}
    d_{t}(f)= \int_{0}^{\pi/2}f\Big(\frac{1}{\cos^2v}\Big)\frac{\xi_{\lambda}(v)}{\cos^2v}dv,
\end{equation}
with $\lambda=1/4+t^2$.
The function $\xi_{\lambda}$ is the solution to the differential equation
$$\xi_{\lambda}''(v)+\frac{\lambda}{\cos^2v}\xi_{\lambda}(v)=0,\quad  v\in \Big(-\frac{\pi}{2},\frac{\pi}{2}\Big),$$
with initial conditions $\xi_{\lambda}(0)=1$ and $\xi_{\lambda}'(0)=0$.
\end{theorem}
We can write $\xi_{\lambda}(v)$ in terms of a sum of Legendre functions, see \cite[p. 5]{CP} and \eqref{Leg}:
$$ \xi_{\lambda}(v)= \frac{1}{2\sqrt{\pi}}\Gamma\Big(\frac{s+1}{2}\Big)\Gamma\Big(1-\frac{s}{2}\Big)(P_{s-1}(i\tan v)+P_{s-1}(-i\tan v)),$$ and, after substituting $x=\tan v$, the transform takes the form
\begin{equation}\label{d(f)}
d_t(f)= \frac{1}{2\sqrt{\pi}}\Gamma\Big(\frac{s+1}{2}\Big)\Gamma\Big(1-\frac{s}{2}\Big) \int_0^{\infty} f(x^2+1)(P_{s-1}(ix)+P_{s-1}(-ix))dx.
\end{equation}
The Huber transform $d_{t}(f)$, is the analogue to the Selberg--Harish-Chandra transform. For the exact relation between them, see \cite{Lsieve}.
Huber showed that, by choosing $f$ to be the characteristic function on the interval $\Big[0,\sqrt{X^2-1}\Big]$, we get 
$A(f)(z)= N(t,\mathfrak{T},z)$. The main result is the following theorem:
\begin{theorem}\rm \cite[p. 983]{CP}\label{CPthm}
We have
$$N(t,\mathfrak{T},z)=\sum_{1/2<s_j\le1}A(s_j)\hat{u}_ju_j(z)X^{s_j}+O\big(X^{2/3}\big),$$
where $A(s)$ is a combination of Gamma functions and special functions, see \cite[Eq. 1.5]{CP}. 
\end{theorem}
Huber also proved the following important Lemma about the periods $\hat{u}_j$:
\begin{lem}\label{periods}
For the sequence of period integrals $\{\hat{u}_j\}_{j=0}^{\infty}$, the following bound holds:
$$\sum_{|t_j|\le T}|\hat{u}_j|^2\ll T.$$
\end{lem}
Tsuzuki \cite[Th. 1, p.~2388]{Tsuz} improved this by proving
$$ \sum_{0 \le t_j < T}|\hat{u}_j|^2 \sim  \frac{\hbox{len}(l)}{\pi}T, \quad T \to \infty .$$
Let $$E(X,z):=N(t,\mathfrak{T},z)-\sum_{1/2<s_j\le1}A(s_j)\hat{u}_ju_j(z)X^{s_j}$$
be the error term in the hyperbolic conjugacy class problem. Chatzakos--Petridis also showed the following results about the second moment of the error term $E(X,z)$ for $\Gamma$ cocompact or cofinite:
\begin{theorem}\rm \cite[Th. 1.2]{CP}\label{secondmomC}
The following bound for the error term holds 
$$ \frac{1}{X}\int_X^{2X}|E(x,z)|^2dx \ll X\log^2X.$$
\end{theorem}

This theorem suggests that the correct order of growth for the error term $E(X,z)$ is $O\big(X^{1/2+\epsilon}\big)$.\par

Parkkonen and Paulin used ergodic methods and, more specifically, the geodesic flow to study the hyperbolic lattice counting problem in conjugacy classes in \cite{PP}. They gave an asymptotic for the counting of common perpendicular arcs in negative curvature, see \cite[Th.~1, p.~901] {parkkonen_paulin_2017}.

\section{A relative trace formula}\label{A relative trace formula}
In this section we investigate a relative trace formula suitable for the problem of counting in the double coset $H_1 \backslash \Gamma / H_1$. \\
We let $\Gamma$ be cocompact. Moreover, we assume that $\Gamma$ has no elements with zeros on the diagonal. For simplicity, we assume that $L$ is a primitive closed geodesic of $\Gamma \backslash\mathbb{H}$. By conjugation we can assume that the axis of the closed geodesic $L$ is the imaginary axis $I$, so that $L= H_1\backslash I$, where $H_1$ is a hyperbolic subgroup of $\Gamma$ of the form 
$$ H_1 =  \Biggl \langle \begin{bmatrix}
m & 0 \\
0 & m^{-1}
\end{bmatrix} \Biggr \rangle, \quad m>1. $$ 
Let $l$ represent the closed geodesic $L$ on the imaginary axis. Here $m^2$ is the norm of the primitive closed geodesic $L$.
It is known that $\hbox{len}(l)=2\log m$. Let $k:\big(-\pi/2,\pi/2\big) \rightarrow \mathbb{R}$ be a test function that depends only on $v$, the angle defined in \eqref{uv}. In practice we also assume that $k$ is even. We consider its automorphization with respect to $\Gamma$:
\begin{equation}\label{Kz}
    K(z):= \sum_{\gamma \in H_1 \backslash \Gamma} k(v(\gamma z)).
\end{equation}
Firstly, we develop the geometric side of the relative trace formula by integrating $K$ over the geodesic segment $l$. \\
We split the sum in \eqref{Kz} into the identity coset $H_1$ and the rest of the cosets. For the coset $H_1$ and $z=iy \in l$ we have that $k(v(\gamma z))=k(v(z))=k(0)$. Hence, we compute
\begin{equation}\label{k0}
    \int_{l} K(z)ds= k(0)\hbox{len}(l) + \int_{l} \sum_{\gamma \in H_1 \backslash \Gamma-H_1} k(v(\gamma z))ds.
\end{equation}
For the second sum we have
\begin{align}\label{geK}
    \begin{split}
        \int_{l} \sum_{\gamma \in H_1 \backslash \Gamma-H_1} k(v(\gamma z))ds &= \sum_{\gamma \in H_1\backslash \Gamma - H_1 /H_1} \sum_{\gamma_0 \in H_1} \int_lk(v(\gamma \gamma_0 z))ds \\
        &= \sum_{\gamma \in H_1\backslash \Gamma - H_1 /H_1} \int_I{k(v(\gamma z))ds},
    \end{split}
\end{align}
using that $H_1$ is the stabilizer of $l$. We notice that the action  of $H_1$ on $l$ covers the whole imaginary axis $I$.\\
We apply Theorem \ref{hub} for $$k(v(z)):=f\Bigg(\frac{1}{\cos^2(v(z))}\Bigg),$$
with appropriate choice of $f$.
In order to ensure convergence of the series $K(z)=A(f)(z)$ we assume that either $f \in C_0^{*}[1,\infty)$ or $f$ is in the Schwartz class. In the first case the series is finite, while in the latter case we can write $K(z)$ as a Stiltjes integral:
\begin{equation}\label{Stil}
    \sum_{\gamma \in H_1\backslash \Gamma} f\Bigg(\frac{1}{\cos^2v(\gamma z)}\Bigg)= \int_0^{\pi/2}f\Bigg(\frac{1}{\cos^2v}\Bigg)d\widetilde{N}(v,z),
\end{equation}
where 
$$\widetilde{N}(V,z):=\#\big\{\gamma \in H_1 \backslash \Gamma \: \big| \: v(\gamma z) \le V \big\}.$$ For $X=1/\cos{v}$ we have 
$$ \widetilde{N}(v,z)=N(t,\mathfrak{T},z) \ll X,$$
using Theorem \ref{CPthm}. Applying integration by parts on the integral in \eqref{Stil}, we see that $K(z)$ converges, assuming that $f$ and its derivatives are rapidly decreasing.

By Theorem \ref{hub}  we get
\begin{equation}\label{spK}
 K(z)= \sum_{j}2d_{t_j}(f)\hat{u}_ju_j(z).
\end{equation}
The spectral side of the relative trace formula comes from integrating \eqref{spK} over $l$. Equating \eqref{k0} and \eqref{geK} with \eqref{spK} gives:
$$ f(1)\cdot \hbox{len}(l) + \sum_{\gamma \in H_1\backslash \Gamma - H_1 /H_1} \int_{I}  f\Bigg(\frac{1}{\cos^2v(\gamma z)}\Bigg)ds = \sum_{j}2d_{t_j}(f)\hat{u}_j^2.$$
In order to evaluate 
$$N(X,l)= \sum_{\substack{\gamma \in H_1\backslash \Gamma/H_1 \\ B(\gamma)<X}}1, $$
where $B(\gamma)=|ad+bc|$,
we analyse further the geometric side, specifically the integral

$$J_{\gamma}^I(f)= \int_{I}f\Bigg(\frac{1}{\cos^2v(\gamma z)}\Bigg)ds = \int_{0}^{\infty}f\Bigg(\frac{1}{\cos^2v(\gamma \cdot iy)}\Bigg) \frac{dy}{y}.$$
We express $J_{\gamma}^I(f)$ in terms of the matrix entries of $\gamma$. We show the following result.
\begin{lem}\label{cosv}
Let $z=iy$ and $\gamma= \begin{bmatrix}
a & b \\
c & d
\end{bmatrix} $, then 
$$ \frac{1}{\cos ^2v(\gamma z)}= \frac{(a^2y^2+b^2)(c^2y^2+d^2)}{y^2}.$$

\end{lem}
\begin{proof}
The proof is an elementary calculation using $\cos v(z)=y/|z|$.

\end{proof}
Lemma \ref{udist} was also proved with small typos in \cite{MMW}. For completeness we present its proof:
\begin{proof}[Proof of Lemma \ref{udist}]
From \cite[p.~9, Eq.~9]{MMW} we have that 
$$4u(\gamma \cdot ix,iy)=\frac{a^2x}{y}+\frac{b^2}{xy}+c^2xy+d^2\frac{y}{x}-2.$$
Let $h(x,y)=4u(\gamma \cdot ix,iy)$. We want to find the minimum value of $h$, hence we compute its gradient:
\begin{equation}\label{h_x}
    h_x(x,y)=0 \implies y^2=\frac{a^2x^2-b^2}{d^2-c^2x^2}.
\end{equation}
Using this value of $y$ in the equation of $h_y(x,y)=0$, we get the solutions
$$x^4=\frac{b^2d^2}{a^2c^2}.$$
Working in a similar way we find that the solutions to the equation $h_x(x,y)=0$ are 
$$y^4=\frac{a^2b^2}{c^2d^2}.$$
Let $x_{\min}^2=|(bd)/(ac)|$ and $y_{\min}^2=|(ab)/(cd)|$. Then we compute
$$h(x_{\min},y_{\min})=2|ad|+2|bc|-2.$$

Suppose that $abcd<0$. Then, there are two cases: either (i)\:$ad>0$ and $bc<0$, or (ii)\:$ad<0$ and $bc>0$. The second case is not possible because $ad-bc=1$. In the first case we have $h(x_{\min},y_{\min})=0$.

On the other hand, if $abcd>0$, then either (i)\:$ad>0$ and $bc>0$, or (ii)\:$ad<0$ and $bc<0$. In both cases we have
$$h(x_{\min},y_{\min})=2|ad+bc|-2 \; .$$
Moreover, the point $(x_{\min},y_{\min})$ minimizes $h$ because $h(x,y)$ tends to $\infty$, as $x$ or $y\to 0^+$  or $\infty$. More precisely, there exist constants $C,M \in \mathbb{R}_{>0}$ such that for all $N\ge M$ and for all $(x,y)\notin[1/N,N]^2, h(x,y)\ge C\cdot N$.
\end{proof}
Using Lemma \ref{cosv}, we see that the integral $J_\gamma^I(f)$ takes the form
$$J_{\gamma}^I(f)=\int_{0}^{\infty} f\Bigg(a^2c^2y^2+\frac{b^2d^2}{y^2}+a^2d^2+b^2c^2\Bigg)\frac{dy}{y}.$$
Similarly to \cite[p.~9]{MMW}, we introduce the change of variables $y=e^t$. We have
$$ J_{\gamma}^I(f)= \int_{-\infty}^{\infty}f\big(a^2c^2e^{2t}+b^2d^2e^{-2t}+a^2d^2+b^2c^2\big)dt.$$
Let $p=a'e^{2t}+b'e^{-2t}$, with $a'=a^2c^2$ and $b'=b^2d^2$, so that $dp=\big(2a'e^{2t}-2b'e^{-2t}\big)dt$. We compute:
$$dt=\frac{dp}{2a'e^{2t}-2b'e^{-2t}}= \frac{dp}{2\sqrt{p^2-4a'b'}},$$
since
\begin{align}
    p^2-4a'b' &= \big(a'\big)^2e^{4t}+\big(b'\big)^2e^{-4t}+2a'b'-4a'b'= \big(a'e^{2t}-b'e^{-2t}\big)^2\notag \\
    &= \big(a^2c^2e^{2t}-b^2d^2e^{-2t}\big)^2 \notag.
\end{align}
In order to find the lower limit of the integral after the change of variables to $p$, we need to compute
$\min_{\substack{t \in \mathbb{R}}}(a^2c^2e^{2t}+b^2d^2e^{-2t})$. Using the first derivative test, we end up with:
\begin{align} \notag
    2a^2c^2e^{2t}-2b^2d^2e^{-2t} =0 \Leftrightarrow
    e^{4t} = \frac{b^2d^2}{a^2c^2} \Leftrightarrow
    t = \frac{1}{2}\log \Big|\frac{bd}{ac}\Big|. \notag
\end{align}
We conclude that
$$J_{\gamma}^I(f)= \int_{2|abcd|}^{\infty} f\big(p+a^2d^2+b^2c^2\big) \frac{dp}{\sqrt{p^2-4a^2b^2c^2d^2}}.$$
Assuming that $\Gamma$ is torsion-free, we have that $B(\gamma)\neq 1$ (see \cite[Lemma 8]{Tsuz}).
Suppose that $\gamma$ satisfies that $abcd>0$, namely $B(\gamma)>1$. Since $ad-bc=1$, we immediately get that $a^2d^2+b^2c^2= 2abcd+1$. Now let $q=p+2abcd$, which gives us
$$ J_{\gamma}^I(f)=\int_{4abcd}^{\infty}\frac{f(q+1)}{\sqrt{q(q-4abcd)}}dq.$$
In order to match our test function with the ones defined in \cite{Hu98} and \cite{CP}, we finally introduce the change of variables $x^2=q$. 
The integral becomes
$$ J_{\gamma}^I(f)=2 \int_{\sqrt{4abcd}}^{\infty}\frac{f(x^2+1)}{\sqrt{x^2-4abcd}}dx.$$
We recall at this point that  $$B(\gamma)=ad+bc.$$
We can relate the quantity $4abcd$ with $B(\gamma)$. We have:
\begin{align}
  \begin{split}
 B(\gamma)^2-1= 4abcd, \notag
  \end{split}
\end{align}
which leads to
$$ J_{\gamma}^I(f)= 2\int_{\sqrt{B(\gamma)^2-1}}^{\infty}\frac{f(x^2+1)}{\sqrt{x^2-(B(\gamma)^2-1)}}dx.$$
This is the contribution to the geometric side of our relative trace formula for the double coset of $\Gamma$ in $H_1 \backslash \Gamma / H_1$. \\
Now, suppose that $abcd<0$, i.e. $B(\gamma)<1$. Following the same process we find that 
$$ J_{\gamma}^I(f)= 2\int_{0}^{\infty}\frac{f(x^2+1)}{\sqrt{x^2-(B(\gamma)^2-1)}}dx.$$
Let 
\begin{equation}\label{qz}
q(z)= 2\int_{\sqrt{z^2-1}}^{\infty}\frac{f(x^2+1)}{\sqrt{x^2-(z^2-1)}}dx,
\end{equation}
and
\begin{equation}\label{qtilde}
\widetilde{q}(z)=2\int_{0}^{\infty}\frac{f(x^2+1)}{\sqrt{x^2-(z^2-1)}}dx.
\end{equation}
We have proved the following formula:
\begin{theorem}[Relative Trace Formula]\label{traceformula}
Let $\Gamma$ be torsion-free and cocompact. Let also $f \in C_0^{*}[1,\infty)$ or $f:[1, \infty) \to \mathbb{R}$ be in the Schwartz class. For $q$ as in \eqref{qz} and $\widetilde{q}$ as in \eqref{qtilde} and $d_t(f)$ as in \eqref{d(f)}, we have:
\begin{align}
    \begin{split}
        f(1)\hbox{\rm len}(l)&+\sum_{\substack{\gamma \in H_1 \backslash \Gamma -H_1 /H_1 \\ B(\gamma)<1}}\widetilde{q}(B(\gamma))
        +\sum_{\substack{\gamma \in H_1 \backslash \Gamma -H_1 /H_1 \\ B(\gamma)>1}}q(B(\gamma))=\sum_{j}2d_{t_j}(f)\hat{u}_j^2.
    \end{split}
\end{align}
\end{theorem}
\begin{remark}
    The case of $\Gamma$ a cofinite subgroup is discussed in \cite{Voskou}. 
\end{remark}

\section{The geometric side of the trace formula}\label{The geometric side of the trace formula}
For the counting problem $N(X,l)$ we want to choose a test function $f$ so that $q(B(\gamma))=1$, whenever $B(\gamma)<X$, namely $q$ is the characteristic function on the interval $[0,X]$. However, this is not continuous and we will need to use smoothings of the characteristic function instead. Motivated by the Selberg--Harish-Chandra transform, we want to write $q$ as a Weyl integral (see \cite[Chapter XIII]{Ederlyi} for the definition of such integrals). Using the substitution $u=x^2+1$, we rewrite \eqref{qz} as
$$ q(z)=\int_{z^2}^{+\infty} \frac{f(u)}{\sqrt{u-1}}\frac{du}{\sqrt{u-z^2}}.$$
We set $$F(u)= \frac{f(u)}{\sqrt{u-1}}, $$
and
\begin{equation}\label{gv}
    g(v)= \int_{v}^{\infty} \frac{F(u)}{\sqrt{u-v}} du.
\end{equation} 
From \cite[Eq. 1.64]{IW} we can recover $F$ and $f$. We get:
\begin{equation}\label{Fu}
    F(u)= -\frac{1}{\pi}\int_u^{\infty}\frac{1}{\sqrt{v-u}} dg(v).
\end{equation}
The outcome is 
$$q(z)=g(z^2).$$
Now, we choose $g$ to be piecewise linear so that $q$ is a smoothing of the characteristic function. Let $H=Y^2+2YX$. We choose \begin{equation}\label{defg+}
 g(y)= 
\begin{cases}
1, & y \le X^2, \\
\displaystyle{\frac{(X+Y)^2-y}{H}}, & X^2 \le y \le (X+Y)^2, \\
0, & y \ge (X+Y)^2.
\end{cases} 
\end{equation}
We now compute $F$ using \eqref{Fu}.\\
{\it Case 1}: When $u>(X+Y)^2$, then $F(u)=0.$\\
{\it Case 2}: For $u<X^2$ we get 
$$ F(u)= -\frac{1}{\pi} \int_{X^2}^{(X+Y)^2} \frac{1}{\sqrt{v-u}}g'(v)dv,$$
since the derivative of $g$ is $0$ for $u<X^2$. We compute:
\begin{equation}
\begin{split}
    F(u) &= \frac{1}{\pi H} \int_{X^2}^{(X+Y)^2} \frac{1}{\sqrt{v-u}}dv=  \frac{2}{\pi H}\Big( \sqrt{(X+Y)^2-u}-\sqrt{X^2-u}\Big). \notag
\end{split}
\end{equation}
Hence, for $u<X^2$, we get $$f(u)= \frac{2}{\pi H}\sqrt{u-1}\Big(\sqrt{(X+Y)^2-u}-\sqrt{X^2-u}\Big).$$
Since in \eqref{qz} $f$ appears as $f(x^2+1)$, we write
$$ f(x^2+1)= \frac{2}{\pi H}x \Big(\sqrt{(X+Y)^2-x^2-1}-\sqrt{X^2-x^2-1}\Big).$$
After setting $a=\sqrt{X^2-1}$ and $A=\sqrt{(X+Y)^2-1}$ the previous expression becomes
$$ f(x^2+1)= \frac{2}{\pi H}x \Big(\sqrt{A^2-x^2}-\sqrt{a^2-x^2}\Big).$$\\
{\it Case 3}:
When $X^2 \le u \le (X+Y)^2$ or, equivalently, $a \le x \le A$, for $u=x^2+1$, we get
$$F(u)= -\frac{1}{\pi} \int_{u}^{(X+Y)^2} \frac{1}{\sqrt{v-u}}g'(v)dv,$$
and, with a similar computation:
$$ f(x^2+1)= \frac{2}{\pi H}x\sqrt{A^2-x^2}.$$
We conclude that
\begin{equation} \label{f+}
f(x^2+1)= 
\begin{cases}
{\frac{2}{\pi H}x \Big(\sqrt{A^2-x^2}-\sqrt{a^2-x^2}\Big)}, & x \le a, \\ 
{\frac{2}{\pi H}x\sqrt{A^2-x^2}}, & a \le x \le A, \\
0, & \text{else}.
\end{cases} 
\end{equation}
Since the corresponding $q$ is an overestimate of the characteristic function of the interval $[0,X]$ we denote this function as $q^{+}$, and, similarly, we denote $f$ by $f^{+}$ and $g$ by $g^{+}$. After defining 
\begin{equation}\label{defg-} g^{-}(y)= 
\begin{cases}
1, & y \le X^2-H, \\
\displaystyle{\frac{X^2-y}{H}}, & X^2-H \le y \le X^2, \\
0, & y \ge X^2,
\end{cases} 
\end{equation}
and repeating the same process, we obtain the test function
\begin{equation} \label{f-}
f^{-}(x^2+1)=
\begin{cases}
{\frac{2}{\pi H}x \Big(\sqrt{a^2-x^2}-\sqrt{T^2-x^2}\Big)}, & x \le T, \\
{\frac{2}{\pi H}x\sqrt{a^2-x^2}}, & T \le x \le a, \\
0, & \text{else},
\end{cases} 
\end{equation}
where $T=\sqrt{X^2-H-1}$. 

Let $g^{+}=g$ and $g^{-}$  be the functions defined in \eqref{defg+} and \eqref{defg-} respectively. Given these choices of $f^{-},f^{+}$, let $\widetilde{g}^{-}(z^2)=\widetilde{q}^{-}(z)$ and $\widetilde{g}^{+}(z^2)=\widetilde{q}^{+}(z)$ be the corresponding functions for the exceptional terms, namely
$$\widetilde{g}^{+}(z^2)=2\int_0^{\infty}\frac{f^{+}(x^2+1)}{\sqrt{x^2-(z^2-1)}}dx,$$ and
$$\widetilde{g}^{-}(z^2)=2\int_0^{\infty}\frac{f^{-}(x^2+1)}{\sqrt{x^2-(z^2-1)}}dx.$$
We show the following result for sums of exceptional terms including $\widetilde{g}^{+}$ and $\widetilde{g}^{-}$:
\begin{lem}\label{Exc}
We have 
\begin{equation}
\sum_{\substack{{\gamma \in H_1 \backslash \Gamma /H_1} \\ B(\gamma) \leq 1 }}\widetilde{g}^{+}(B(\gamma)^2)=O(1)
\end{equation}
and 
\begin{equation}
\sum_{\substack{{\gamma \in H_1 \backslash \Gamma /H_1} \\ B(\gamma) \leq 1 }}\widetilde{g}^{-}(B(\gamma)^2)=O(1).
\end{equation}
\end{lem}
\begin{proof}
For $B(\gamma)\le 1$ we notice that 
$$ \widetilde{g}^{+}\big(B(\gamma)^2\big)\le \int_0^{\infty}\frac{f^{+}(x^2+1)}{x}dx = \frac{2}{\pi H}\Bigg(\int_0^{A}\sqrt{A^2-x^2}dx-\int_0^a\sqrt{a^2-x^2}dx\Bigg).$$
After the change of variables $x=A\sin t$ for the first integral and $x=a\sin t$ for the second integral we see that
\begin{equation}\label{g+}
    \widetilde{g}^{+}\big(B(\gamma)^2\big)\le \frac{2}{\pi H}\frac{(A^2-a^2)\pi}{4} \le 1.
\end{equation}
A similar calculation gives that $\widetilde{g}^{-}\big(B(\gamma)^2\big) \le 1$.
By discreteness and inequality \eqref{g+} we conclude that the sums for the exceptional terms are $O(1)$.
\end{proof}

Moreover, we arrive at the following estimate:
\begin{proposition}
Let $g^{+}=g$ and $g^{-}$  be the functions defined in \eqref{defg+} and \eqref{defg-} respectively. Then we have the following estimates for $N(X,l)$:

    \begin{equation}\label{Nineq}
\sum_{\substack{{\gamma \in H_1 \backslash \Gamma /H_1} \\ B(\gamma) \leq 1 }}1+\sum_{\substack{{\gamma \in H_1 \backslash \Gamma -H_1 /H_1} \\ B(\gamma)>1 }}g^{-}\big(B(\gamma)^2\big) \le N(X,l) \le \sum_{\substack{{\gamma \in H_1 \backslash \Gamma /H_1} \\ B(\gamma)\leq1 }}1+\sum_{\substack{{\gamma \in H_1 \backslash \Gamma -H_1 /H_1} \\ B(\gamma)>1 }}g^{+}\big(B(\gamma)^2\big).
\end{equation}
\end{proposition}
\begin{proof}

We note that $f^{+}(1)=f^{-}(1)=0$.
From \eqref{defg+} and \eqref{defg-}, since $0\le g^{-}(y)\le 1$ in $\big[0,X^2\big]$ and $g^{-}(y)=0$ for $y>X^2$, we clearly have
\begin{equation}\label{g1g}
    \sum_{\substack{{\gamma \in H_1 \backslash \Gamma -H_1 /H_1} \\ B(\gamma)>1 }}g^{-}\big(B(\gamma)^2\big) \le \sum_{\substack{{\gamma \in H_1 \backslash \Gamma -H_1 /H_1} \\1<B(\gamma)<X }}1 \le \sum_{\substack{{\gamma \in H_1 \backslash \Gamma -H_1 /H_1} \\ B(\gamma)>1 }}g^{+}\big(B(\gamma)^2\big).
\end{equation}
By its definition, $N(X,l)$ can be written as 
$$N(X,l)=\sum_{\substack{{\gamma \in H_1 \backslash \Gamma/H_1} \\ B(\gamma)\leq 1 }}1+\sum_{\substack{{\gamma \in H_1 \backslash \Gamma -H_1 /H_1} \\ 1<B(\gamma)<X }}1,$$
so the result follows from \eqref{g1g}. 
\end{proof}

\section{The spectral side of the relative trace formula}\label{The spectral side of the relative trace formula}
We proceed with the analysis of the spectral side of Theorem \ref{traceformula} for $f^{+}$. The computations regarding $f^{-}$ follow in the same way.
The integral transform that appears in the Huber transform in \eqref{d(f)} is
$$\int_0^{\infty}f^{+}(x^2+1)(P_{s-1}(ix)+P_{s-1}(-ix))dx,$$ and, for the specific test function $f^{+}$ that we chose in \eqref{f+}, it becomes
\begin{align}
    & \int_0^{a}\frac{2}{\pi H}x \Big(\sqrt{A^2-x^2}-\sqrt{a^2-x^2}\Big)(P_{s-1}(ix)+P_{s-1}(-ix))\: dx \\
    & + \int_{a}^{A}\frac{2}{\pi H}x\sqrt{A^2-x^2}(P_{s-1}(ix)+P_{s-1}(-ix)) dx.
\end{align}
For $y>0$, let 
\begin{equation}\label{Jsy}
    J_s(y)=\int_0^{\sqrt{y}}x\sqrt{y-x^2}\big(P_{s-1}(ix)+P_{s-1}(-ix)\big)dx.
\end{equation}
The Huber transform of $f^{+}$ can be written as
\begin{equation}\label{Jdiff}
d_t(f^{+})=\pi^{-3/2} \Gamma\Big(\frac{s+1}{2}\Big)\Gamma\Big(1-\frac{s}{2}\Big)\frac{J_s\big(A^2\big)-J_s\big(a^2\big)}{H}.
\end{equation}
We prove the following Lemma for $d_t(f^{+})$:
\begin{lem}\label{mainlem}
For $s=\frac{1}{2}+it$ with $t\notin \mathbb{R}$ and $s\neq 1$, i.e. $1/2<s<1$, we have:
\begin{align}\label{dt3}
    \begin{split}
        d_t(f^{+})&=\frac{\gamma_1(s)}{\pi}\frac{2}{s+2}\frac{A^{s+2}-a^{s+2}}{A^2-a^2}(1+O(X^{-2})) \\
        &+\frac{\gamma_2(s)}{\pi}\frac{A^{3-s}-a^{3-s}}{A^2-a^2}(1+O(X^{-2}))\\
        &+\frac{\gamma_3(s)}{\pi}\frac{A-a}{A^2-a^2}(1+O(X^{-2})),
    \end{split}
\end{align}
where
\begin{align}\label{gamma123}
    \begin{split}
        &\gamma_1(s)=\frac{\pi}{2}\frac{\Gamma(s+1/2)\Gamma((s+1)/2)}{\Gamma(s/2)^3}\frac{4}{s(2s-1)}, \quad \gamma_2(s)=\frac{\pi}{2}\frac{\Gamma((1-2s)/2)\Gamma((2-s)/2)}{\big(\Gamma((1-s)/2)\big)^2\Gamma((5-s)/2)}, \\
        &\text{and} \quad \gamma_3(s)=-\frac{1}{2}\frac{\Gamma((-1-s)/2)\Gamma((s-2)/2)}{\Gamma((1-s)/2)\Gamma(s/2)}.
    \end{split}
\end{align}

\end{lem}

\begin{proof}
By the definition of the Legendre function $P_{s-1}$, see \cite[8.704]{gradshteyn2007}, we have
\begin{align}
    \begin{split}
        J_s\big(A^2\big)=&\int_0^Ax\sqrt{A^2-x^2}\big(P_{s-1}(ix)+P_{s-1}(-ix)\big)dx \\
        =&\int_0^Ax\sqrt{A^2-x^2}\Big({}_2F_1\Big(1-s,s,1;\frac{1-ix}{2}\Big)+{}_2F_1(1-s,s;1;\frac{1+ix}{2}\Big)\Big)dx.
    \end{split}
\end{align}
Applying formulas \cite[9.136.2 and 9.136.3]{gradshteyn2007} for $\alpha=(1-s)/2,\beta=s/2,z=-x^2$ to both hypergeometric functions inside the integral, we get 
\begin{equation}\label{Js}
    J_s\big(A^2\big)=2D(s)\int_0^Ax\sqrt{A^2-x^2}\;  {}_2F_1\Big(\frac{1-s}{2},\frac{s}{2},\frac{1}{2};-x^2\Big)dx,
\end{equation}
where $$D(s)=\displaystyle \frac{\sqrt{\pi}}{\Gamma\big(\frac{2-s}{2}\big)\Gamma\big(\frac{s+1}{2}\big)}.$$
Let $x^2=u$ and then $u=vA^2$ to get
$$J_s(A^2)=A^3D(s)\int_0^1\sqrt{1-v}\;{}_2F_1\Big(\frac{1-s}{2},\frac{s}{2},\frac{1}{2};-vA^2\Big)dv.$$
By \cite[Eq. 7.512.12]{gradshteyn2007} we have 
$$J_s(A^2)=A^3D(s)\frac{\Gamma(3/2)}{\Gamma(5/2)}\:_3F_2\Big(\frac{1-s}{2},\frac{s}{2},1;\frac{1}{2},\frac{5}{2};-A^2\Big) \; .$$
By \cite[Eq. 16.8.8]{NIST} for $q=2, z=-A^2$, and since $$_3F_2\big(a,b,c,d,a;z\big)=\:_2F_1\big(b,c,d;z\big),$$ we get for $s\neq 1/2$ and $s\neq1$
\begin{align}\label{JF3}
    \begin{split}
            J_s(A^2)=&A^3D(s)\frac{\Gamma(3/2)\Gamma(1/2)}{\Gamma((1-s)/2)\Gamma(s/2)}\Bigg(\frac{\Gamma((1-s)/2)\Gamma((2s-1)/2)\Gamma((1+s)/2)}{\Gamma(s/2)\Gamma((4+s)/2)}A^{s-1}\times \\
        &\times \:_2F_1\Big(1-\frac{s}{2},-1-\frac{s}{2},\frac{3-2s}{2},-A^{-2}\Big) \\
        &+\frac{\Gamma(s/2)\Gamma((1-2s)/2)\Gamma(1-s/2)}{\Gamma((1-s)/2)\Gamma((5-s)/2)}A^{-s}\:_2F_1\Big(\frac{s+1}{2},\frac{s-3}{2},\frac{2s+1}{2};-A^{-2}\Big)\\
        &+\frac{\Gamma((-1-s)/2)\Gamma((s-2)/2)}{\Gamma(-1/2)\Gamma(3/2)}A^{-2}\:_3F_2\Big(1,\frac{3}{2},-\frac{1}{2};\frac{s+3}{2},\frac{4-s}{2};-A^{-2}\Big)\Bigg). \\
    \end{split}
\end{align}
We compute the Gamma factors in \eqref{JF3} for each summand: they are $D(s)$ times $\gamma_1(s)(2/(s+2))$, $\gamma_2(s)$, and, $\gamma_3(s)$ respectively, see \eqref{gamma123}.

If $1/2<s<1$, we use the series expansion of the hypergeometric functions in \eqref{Js} to compute that
$$_2F_1\Big(1-\frac{s}{2},-1-\frac{s}{2},\frac{3-2s}{2},-A^{-2}\Big)=1+O\big(A^{-2}\big) \; ,$$
$$_2F_1\Big(\frac{s+1}{2},\frac{s-3}{2},\frac{2s+1}{2};-A^{-2}\Big)=1+O\big(A^{-2}\big) \; ,$$ $$_3F_2\Big(1,\frac{3}{2},-\frac{1}{2};\frac{s+3}{2},\frac{4-s}{2};-A^{-2}\Big)=1+O\big(A^{-2}\big).$$
We use the same methods to compute $J_s(a^2)$.
Since $a,A \sim X$ we have
\begin{align}
    \begin{split}
        d_t(f^{+})&= \frac{1}{2\sqrt{\pi}}\Gamma\Big(\frac{s+1}{2}\Big)\Gamma\Big(1-\frac{s}{2}\Big)\frac{2}{\pi H}\big(J_s\big(A^2\big)-J_s\big(a^2\big)\big) \\
        &=\frac{\gamma_1(s)}{\pi}\frac{2}{s+2}\frac{A^{s+2}-a^{s+2}}{A^2-a^2}\big(1+O\big(X^{-2}\big)\big) \\
        &\quad +\frac{\gamma_2(s)}{\pi}\frac{A^{3-s}-a^{3-s}}{A^2-a^2}\big(1+O\big(X^{-2}\big)\big)\\
        &\quad +\frac{\gamma_3(s)}{\pi}\frac{A-a}{A^2-a^2}\big(1+O\big(X^{-2}\big)\big).
    \end{split}
\end{align}
\end{proof}
We note that, using Stirling's approximation \eqref{Stir}, the gamma functions appearing can be bounded as follows:

\begin{equation}\label{Gammas}
    \gamma_1(s) \ll |t|^{-1/2}\: ,\quad \gamma_2(s) \ll |t|^{-3/2} \quad \text{and} \quad \gamma_3(s) \ll |t|^{-2}.
\end{equation}
More generally we can prove the following result about $d_t(f^{+})$:
\begin{lem}\label{lem19}
Let $f^{+}$ be the function defined in \eqref{f+}, then there exists $\xi \in [a^2,A^2]$ such that
$$d_t(f^{+})= \frac{\Gamma \big(\frac{s+1}{2}\big)\Gamma \big(1-\frac{s}{2}\big)}{\pi^{3/2}}\int_0^{\sqrt{\xi}}\frac{x}{\sqrt{\xi-x^2}}\big(P_{s-1}(ix)+P_{s-1}(-ix)\big)dx.$$
\end{lem}
\begin{proof}

We define $$R_s(y)=\int_0^{\sqrt{y}} G_s(y,x) dx,$$
where  
\begin{equation}\label{Gs}
    G_s(y,x)=x\sqrt{y-x^2}\:_{2}F_1\Big(\frac{1-s}{2},\frac{s}{2},\frac{1}{2};-x\Big).
\end{equation}
By \eqref{Jdiff} and \eqref{Js} we have 
\begin{equation}\label{Rdif}
    d_t(f)=\frac{2}{\sqrt{\pi}}\frac{R_s(A^2)-R_{s}(a^2)}{A^2-a^2}.
\end{equation}
Since $H=A^2-a^2$, we apply the Mean Value Theorem for $R_s$ in \eqref{Rdif} on the interval $[a^2,A^2]$. Note that for $s \in [1/2,1]$ or $s=1/2+it$, $t \in \mathbb{R}$, the hypergeometric function in \eqref{Gs} is real, hence the Mean Value Theorem can be applied. By the Leibniz Rule, we have
\begin{align}\notag
    \begin{split}
        R'_s(y)&= G_s(y,\sqrt{y})(\sqrt{y})'-G_s(y,0)0'+\int_0^{\sqrt{y}}\frac{\partial}{\partial y}G_s(y,x)dx \\
        &= \frac{1}{2}\int_0^{\sqrt{y}}\frac{x}{\sqrt{y-x^2}}\:_{2}F_1\Big(\frac{1-s}{2},\frac{s}{2},\frac{1}{2};-x^2\Big)dx.
   \end{split}
\end{align}
Hence, there is $\xi\in [a^2,A^2] \implies X^2-1<\xi<(X+Y)^2-1$, such that 
\begin{align}\notag
    \begin{split}
        d_{t}(f^{+})&=\frac{2}{\sqrt{\pi}}R'_s(\xi)=\frac{1}{\sqrt{\pi}}\int_0^{\sqrt{\xi}}\frac{x}{\sqrt{\xi-x^2}}\:_{2}F_1\Big(\frac{1-s}{2},\frac{s}{2},\frac{1}{2};-x^2\Big)dx\\
       &= \frac{\Gamma \big(\frac{s+1}{2}\big)\Gamma \big(1-\frac{s}{2}\big)}{\pi^{3/2}}\int_0^{\sqrt{\xi}}\frac{x}{\sqrt{\xi-x^2}}\big(P_{s-1}(ix)+P_{s-1}(-ix)\big)dx.
    \end{split}
\end{align}
\end{proof}
Using Lemmata \ref{mainlem} and \ref{lem19} we prove the main estimate for the Huber transform of $f^{+}$ as follows: \begin{proposition}\label{mainprop}
\begin{enumerate}
   \item[(i)] For $s=\frac{1}{2}+it$ with $t\notin \mathbb{R}$, i.e. $1/2<s\le1$, we have 

\begin{align}
    \begin{split}
        d_t(f^{+})&=\frac{\gamma_1\big(\frac{1}{2}+it\big)}{\pi}X^{\frac{1}{2}+it}+\frac{\gamma_2\big(\frac{1}{2}+it\big)}{\pi}\Big(\frac{5}{2}-it\Big)X^{\frac{1}{2}-it} \\
        &\quad+O\Big(X^{s-1}Y+X^{-s}Y+Y\big),
    \end{split}
\end{align}
    
    with $\gamma_1(s),\gamma_2(s)$ given in \eqref{gamma123}.
    \item[(ii)] Let $t\in \mathbb{R}$ and $t\neq 0$, then the Huber transform of $f^{+}$ satisfies the bound
    $$d_t(f^{+})=X^{1/2}O\big(\min \{|t|^{-1/2},|t|^{-3/2}X/Y\}\big).$$
    \item[(iii)] Let $t\in \mathbb{R}$ with $0\neq t \le X^2/Y^2$. Then $d_t(f^{+})$ can be written in the form
    $$ d_t(f^{+})= a(t,X/Y)X^{\frac{1}{2}+it}+b(t,X/Y)X^{\frac{1}{2}-it}+O\big(|t|^{-3/2}\big),$$ where $$a(t,X/Y),b(t,X/Y)=O\big(\min\{|t|^{-1/2},|t|^{-3/2}X/Y\}\big).$$
   
    \item[(iv)] For $t=0$, we have $$d_0(f^{+}) \ll X^{1/2}\log X.$$
\end{enumerate}
\end{proposition}
\begin{proof}

(i) Firstly, we estimate $\big(A^{s+2}-a^{s+2}\big)/(A^2-a^2)$. We consider the real and the imaginary part of this quantity and apply the Mean Value Theorem to each part. 
Since $s=1/2+it$ we have $$\Re\Big((x^2)^{(s+2)/2}\Big)=(x^2)^{5/4}\cos\Big(\frac{t}{2}\log(x^2)\Big),$$ and
$$\Im\Big((x^2)^{(s+2)/2}\Big)=(x^2)^{5/4}\sin\Big(\frac{t}{2}\log(x^2)\Big).$$
By the Mean Value Theorem, we can find  $\xi\in[a^2,A^2]$ such that
\begin{align}
    \begin{split}
        \frac{\Re\big(A^{s+2}\big)-\Re\big(a^{s+2}\big)}{A^2-a^2}&=\frac{5}{4}\xi^{1/4}\cos\Big(\frac{t}{2}\log \xi\Big)-\frac{t}{2}\xi^{1/4}\sin\Big(\frac{t}{2}\log \xi\Big)\\
       &=\frac{5}{4}\Re\big(\xi^{s/2}\big)-\frac{t}{2}\Im\big(\xi^{s/2}\big).
    \end{split}
\end{align}
Similarly, there exists $\zeta \in [a^2,A^2]$ such that
\begin{align}
    \begin{split}
       \frac{\Im\big(A^{s+2}\big)-\Im\big(a^{s+2}\big)}{A^2-a^2}&=-\frac{5}{4}\zeta^{1/4}\sin\Big(\frac{t}{2}\log \zeta\Big)+\frac{t}{2}\zeta^{1/4}\cos\Big(\frac{t}{2}\log \zeta\Big) \\
        &=-\frac{5}{4}\Re\big(\zeta^{s/2}\big)+\frac{t}{2}\Im\big(\zeta^{s/2}\big).
    \end{split}
\end{align}
\begin{align}
    \begin{split}
       \frac{\Im\big(A^{s+2}\big)-\Im\big(a^{s+2}\big)}{A^2-a^2}&=\frac{5}{4}\zeta^{1/4}\sin\Big(\frac{t}{2}\log \zeta\Big)+\frac{t}{2}\zeta^{1/4}\cos\Big(\frac{t}{2}\log \zeta\Big) \\
       &=\frac{5}{4}\Im\big(\zeta^{s/2}\big)+\frac{t}{2}\Re\big(\zeta^{s/2}\big).
   \end{split}
\end{align}
We notice that $\xi^{s/2},\zeta^{s/2} = X^s+O\big(|s|\big|X^{s-1}\big|Y\big)$. Since $1/2<s<1$ it follows that

$$\frac{A^{s+2}-a^{s+2}}{A^2-a^2}= \frac{5}{4}X^s+i\frac{t}{2}X^s+O\big(X^{s-1}Y\big).$$
We conclude that
\begin{equation}\label{s/2}
   \frac{A^{s+2}-a^{s+2}}{A^2-a^2}=\frac{s+2}{2}\Big(X^s+O\big(X^{s-1}Y\big)\Big).
\end{equation}
Similarly we have that 
\begin{equation}\label{1-s}
   \frac{A^{3-s}-a^{3-s}}{A^2-a^2}=\frac{3-s}{2}\Big(X^{1-s}+O\big(X^{-s}Y\big)\Big).
\end{equation}

For the last summand in \eqref{dt3} we use Stirling's approximation to conclude that it is $O(1)$.
Hence, using \eqref{s/2} and \eqref{1-s}, for $1/2<s<1$ we write \eqref{dt3} as:
\begin{align}
    \begin{split}
        d_t(f^{+})&= \frac{\gamma_1(s)}{\pi}X^s+\frac{\gamma_2(s)}{\pi}\frac{3-s}{2}X^{1-s}\\
        &\quad +O\Big(X^{s-1}Y+X^{-s}Y\Big).
    \end{split}
\end{align}
If $s=1$, i.e. $t=-i/2$, from \cite[Eq.8.711.1] {gradshteyn2007} we have that $P_0(\pm ix)=1$. We see that $d_{-i/2}(f)=(2/\pi)X+O(Y)$ and the result follows.

(ii) Suppose that $t\in \mathbb{R}, t\neq 0$ and let 
$$G_s(x)=x^{(s+2)/2}\: _2F_1\Big(1-\frac{s}{2},-1-\frac{s}{2},\frac{3-2s}{2},-x^{-1}\Big).$$
From \eqref{Jdiff} and \eqref{JF3} we have that 

\begin{align}
    \begin{split} \label{4sum}
         d_t(f^{+})&=\frac{\gamma_1(s)}{\pi}\frac{2}{s+2}\frac{G_s(A^2)-G_s(a^2)}{H} \\
         &+\frac{\gamma_2(s)}{\pi}\frac{G_{1-s}(A^2)-G_{1-s}(a^2)}{H}\\
         &+\frac{\gamma_3(s)}{\pi H}\Bigg(A\: _3F_2\Big(1,\frac{3}{2},-\frac{1}{2};\frac{s+3}{2},\frac{4-s}{2};-A^{-2}\Big)\\
        &-a\: _3F_2\Big(1,\frac{3}{2},-\frac{1}{2};\frac{s+3}{2},\frac{4-s}{2};-a^{-2}\Big)\Bigg).
    \end{split}
\end{align}
At the moment, we cannot use the hypergeometric series expansion for large $t$ for the hypergeometric function appearing in $G_s(A^2)$ and derive a bound because $s$ appears in the first two entries. To solve that, we first use \eqref{9.137} to get
\begin{align}
    \begin{split}
        _2F_1\Big(1-\frac{s}{2},-1-\frac{s}{2},\frac{3-2s}{2},-A^{-2}\Big)=&\frac{2}{3-2s}\Big(\frac{3-2s}{2}\: _2F_1\Big(2-\frac{s}{2},-1-\frac{s}{2},\frac{3-2s}{2},-A^{-2}\Big)\\
    &-(1+\frac{s}{2})A^{-2}\:_2F_1\Big(2-\frac{s}{2},-\frac{s}{2},\frac{5}{2}-s,-A^{-2}\Big)\Big).
    \end{split}  
\end{align}
We continue by applying \eqref{9.133}
to get
\begin{equation}
    \frac{2}{3-2s}\Big(\frac{3-2s}{2}\:_2F_1\Big(4-s,-2-s,\frac{3}{2}-s,w\Big)-(1+s/2)A^{-2}\:_2F_1\Big(4-s,-s,\frac{5}{2}-s,w\Big)\Big),
\end{equation}
with $-A^{-2}=4w(1-w)$. Finally, by \eqref{9.131} we write
\begin{align}
\begin{split}
    _2F_1\Big(1-\frac{s}{2},-1-\frac{s}{2},\frac{3-2s}{2},-A^{-2}\Big)=&(1-w)^{-1/2+s}\:_2F_1\Big(-\frac{5}{2},\frac{7}{2},\frac{3}{2}-s,w\Big)\\
    &-(1+s/2)\frac{2}{3-2s}A^{-2}(1-w)^{-3/2+s}\:_2F_1\Big(-\frac{3}{2},\frac{5}{2},\frac{3}{2}-s,w\Big).
\end{split}   
\end{align}
We apply now the hypergeometric series expansion to get
\begin{align}
\begin{split}
    _2F_1\Big(1-\frac{s}{2},-1-\frac{s}{2},\frac{3-2s}{2},-A^{-2}\Big)=&(1-w)^{-1/2+s}\big(1+O(w|s|^{-1}))\\
    &-(1-w)^{-3/2+s}\frac{2+s}{3-2s}A^{-2}(1+O(w|s|^{-1})).
\end{split}
\end{align}
We apply the same process to the hypergeometric function appearing in $G_s(a^2)$.

By Stirling's approximation for $\gamma_1(s)$ (see \eqref{Gammas}) we conclude the bound
$$\frac{\gamma_1(s)}{\pi}\frac{2}{s+2}\frac{G_s(A^2)-G_s(a^2)}{H} \ll |t|^{-3/2}X^{3/2}/Y.$$
On the other hand by the Mean Value Theorem on the expression $\big(A^{s+2}-a^{s+2}\big)/{H}$ and Stirling's approximation for $\gamma_1(s)$ we get the bound 
$$\frac{\gamma_1(s)}{\pi}\frac{2}{s+2}\frac{G_s(A^2)-G_s(a^2)}{H} \ll |t|^{-1/2}X^{1/2}.$$
We repeat the same process for the second summand in \eqref{4sum} in which case we have $1-s$ in place of $s$ for the function $G_{\_}(x)$ and apply Stirling's approximation for the function $\gamma_2(s)$ (see \eqref{Gammas}). We conclude that the second summand satisfies the same bounds as the first one.

We also bound the summands in \eqref{4sum} involving the hypergeometric functions $_3F_2$ by using the hypergeometric series expansion (notice that the factors involving $s$ appear on the denominator) and Stirling's approximation for $\gamma_3$, see \eqref{Gammas}. For $z=A$ or $a$, we have that
$$\frac{\gamma_3(s)}{\pi H}z\: _3F_2\Big(1,\frac{3}{2},-\frac{1}{2};\frac{s+3}{2},\frac{4-s}{2};-z^{-2}\Big)\ll |t|^{-2}Y^{-1} \:, $$
hence those terms are negligible and will not affect the overall bound that we get from the terms involving $G_s(s)$ or $G_{1-s}(x)$ and the result follows.

(iii) We return to \eqref{4sum}. Assuming that $t\le X^2/Y^2$ we use the hypergeometric series to write
\begin{equation}\label{G_s_series}
    G_s(A^2)=A^{s+2}(1+O(X^{-1})) , \quad G_{s}(a^2)=a^{s+2}(1+O(X^{-1}))
\end{equation}
and 
\begin{equation}\label{G_1-s}
    G_{1-s}(A^2)=A^{3-s}(1+O(X^{-1})), \quad G_{1-s}(a^2)=a^{3-s}(1+O(X^{-1})).
\end{equation}
Hence, from \eqref{G_s_series} we get $$\frac{\gamma_1(s)}{\pi}\frac{2}{s+2}\frac{G_s(A^2)-G_s(a^2)}{H}=\frac{\gamma_1(s)}{\pi}\frac{2}{s+2}\frac{A^{s+2}-a^{s+2}}{H}(1+O(X^{-1})).$$
Since $t\le X^2/Y^2$ we use the binomial expansion and the fact that $Y\ge X^{1/2}$ to write 
$$ A^{s+2}=(X+Y)^{(s+2)}(1+O(X^{-1})) \quad \text{and} \quad a^{s+2}=X^{(s+2)}(1+O(X^{-1})).$$
It follows that
$$\frac{\gamma_1(s)}{\pi}\frac{2}{s+2}\frac{G_s(A^2)-G_s(a^2)}{H}=\frac{\gamma_1(s)}{\pi}\frac{2}{s+2}\frac{(X+Y)^{s+2}-X^{s+2}}{H}(1+O(X^{-1})).$$
Since $H=A^2-a^2=Y^2+2YX$ we write
$$\frac{(X+Y)^{s+2}-X^{s+2}}{H}=X^s\frac{(1+Y/X)^{s+2}-1}{Y^2/X^2+2Y/X}.$$
Since $Y\gg X^{1/2}$ we conclude that 
$$\frac{\gamma_1(s)}{\pi}\frac{2}{s+2}\frac{G_s(A^2)-G_s(a^2)}{H}=X^sa(t,X/Y)+O(|t|^{-3/2}),$$
where $$a(t,X/Y)=\frac{\gamma_1(s)}{\pi}\frac{2}{s+2}\frac{(1+Y/X)^{s+2}-1}{Y^2/X^2+2Y/X}.$$
Similarly to the bounds in ii) we use Stirling's approximation and trivial bounds to get that
$$a(t,X/Y) \ll |t|^{-3/2}X^{3/2}/Y.$$
By the binomial expansion we have $(1+Y/X)^{s+2}=1+O(|s|Y/X)$. We apply Stirling's approximation to get the bound
$$a(t,X/Y) \ll |t|^{-1/2}.$$
Using \eqref{G_1-s}, we repeat the same process for the second summand in \eqref{4sum} to get
$$\frac{\gamma_2(s)}{\pi}\frac{G_{1-s}(A^2)-G_{1-s}(a^2)}{H}=X^{1-s}b(t,X/Y)+O(|t|^{-3/2}),$$
for $$b(t,X/Y)=\frac{\gamma_2(s)}{\pi}\frac{(1+Y/X)^{s+2}-1}{Y^2/X^2+2Y/X} = O\big( \min \{|t|^{-1/2},|t|^{-3/2}X/Y\}\big).$$
Same as in iii), we bound the summands in \eqref{4sum} involving the hypergeometric functions ${}_3F_2$ using the hypergeometric series expansion and Stirling's approximation for $\gamma_3(s)$ and the result follows.

(iv) We first need a bound on the sum $P_{-1/2}(ix)+P_{-1/2}(-ix)$. The transformation formula \cite[Eq. 16.8.8]{NIST} for the hypergeometric function used above, is not valid for $s=1/2$, hence we are going to use the integral representation \cite[Eq. 8.713.3]{gradshteyn2007}. Firstly, assume that $x<1$. Then we have 

\begin{equation}\label{x<1}P_{-1/2}(ix)+P_{-1/2}(-ix) \ll \int_0^{\infty}(\cosh^2t+x^2)^{-1/4}dt\ll 1.
\end{equation}
Secondly, assume that $x\ge 1$, then we have
\begin{align}
        \begin{split}
            P_{-1/2}(ix)+P_{-1/2}(-ix) &\ll \int_0^{\infty}(\cosh^2t+x^2)^{-1/4}dt \\
            & \ll x^{-1/2}\int_0^{\infty}\Big(\Big(\frac{\cosh t}{x}\Big)^2+1\Big)^{-1/4}dt.
        \end{split}
\end{align}
Now we set $u=\cosh t\big/x$ and compute:

          $$  \int_0^{\infty}\Big(\Big(\frac{\cosh t}{x}\Big)^2+1\Big)^{-1/4}dt = \int_{1/x}^{\infty}(u^2+1)^{-1/4}\frac{x}{\sqrt{x^2u^2-1}}du $$
         
We split the integration at $u=2$. For 
 $u\ge 2$ we notice that 
    $x\big/\sqrt{x^2u^2-1} \ll u^{-1}$,
    hence we can bound the integration from $2$ to $\infty$ as follows: 
    $$\int_{2}^{\infty}(u^2+1)^{-1/4}\frac{x}{\sqrt{x^2u^2-1}}du \ll \int_{2}^{\infty}(u^2+1)^{-1/4}\frac{1}{u}\ll 1.$$
    We also bound the integral from $1/x$ to $2$ as
    \begin{align}\notag
        \begin{split}
            \int_{1/x}^{2}(u^2+1)^{-1/4}\frac{x}{\sqrt{x^2u^2-1}}du \ll \int_{1/x}^2\frac{x}{\sqrt{x^2u^2-1}}du \ll \log x,
        \end{split}
    \end{align}
    noticing that if we let $xu=\cosh r$, then we have 
    $$\int_{1/x}^2\frac{x}{\sqrt{x^2u^2-1}}du =\int_{0}^{\cosh^{-1}(2x)} dr = \cosh^{-1}(2x).$$
    For $x\ge 1$, we conclude that 
    \begin{equation}\label{x>1}
         P_{-1/2}(ix)+P_{-1/2}(-ix) \ll x^{-1/2}\log x.
    \end{equation}
    It follows from Lemma \ref{lem19} and \eqref{x<1}, \eqref{x>1} that for some $\xi \in [a^2,A^2]$
\begin{align}
        \begin{split}
            d_0(f^{+}) &\ll 
             \int_0^{1}\frac{x}{\sqrt{\xi-x^2}}dx + \int_1^{\sqrt{\xi}}\sqrt{\frac{x}{\xi-x^2}}\log x\:dx \\
            &\ll 1+\int_1^{\sqrt{\xi}}\frac{\xi^{1/4}}{\sqrt{\xi}\sqrt{1-x^2/\xi}}\log( \sqrt{\xi})\:dx \\
             &\ll 1+\xi^{-1/4}\log \xi \int_1^{\sqrt{\xi}}\frac{1}{\sqrt{1-x^2/\xi}}dx \\
            &\ll \xi^{1/4}\log \xi \ll X^{1/2}\log X.
        \end{split}
    \end{align}
\end{proof}
 
\section{Proof of Theorem \ref{mainthm}}\label{Proof of counting}
We will now use our results about the Huber transform $d_t(f^{+})$ to analyse the spectral contribution to the relative trace formula, Theorem \ref{traceformula}, and finally prove Theorem \ref{mainthm}. 
\begin{proof}[Proof of Theorem \ref{mainthm}]
We have by Proposition \ref{mainprop}:
\begin{align}
    \begin{split}
        \sum_j2d_t(f^{+})\hat{u}_j^2=&\sum_{1/2<s_j\le1}\frac{2}{\pi}\gamma_1(s_j)\hat{u}_j^2X^{s_j}+\frac{1}{\pi}(3-s_j)\gamma_2(s_j)\hat{u}_j^2X^{1-s_j}\\
        &+O\Bigg(\sum_{1/2<s_j\le1}\Big(\hat{u}_j^2X^{s-1}Y+{u}_j^2X^{-s}Y\Big)\Bigg)\\
        &+\sum_{0\ne t_j\in \mathbb{R}}2d_t(f^{+})\hat{u}_j^2+O\big(X^{1/2}\log X\big),
    \end{split}
\end{align}
where the last term is due to the estimate for $d_0(f^{+})$ (see Proposition \ref{mainprop}(iv)).
Since the spectrum is discrete, for $s_j$ corresponding to a small eigenvalue, $s_j-\frac{1}{2}$ is bounded away from zero. As the number of small eigenvalues is finite, we get
$$\sum_{1/2<s_j\le1}\hat{u}_j^2\,\,Y=O(Y).$$
For the same reason,
$$\sum_{1/2<s_j\le1}\frac{1}{\pi}\gamma_2(s_j)(3-s_j)\hat{u}_j^2X^{1-s_j}=O\big(X^{1/2}\big).$$
Let $$S(f^{+})=\sum_{0\ne t_j\in \mathbb{R}}2d_t(f^{+})\hat{u}_j^2,$$
then 
\begin{equation}\label{Esf}
\sum_j2d_t(f^{+})\hat{u}_j^2=\sum_{1/2<s_j\le1}\frac{2}{\pi}\gamma_1(s_j)\hat{u}_j^2X^{s_j}+S(f^{+})+O\big(Y+X^{1/2}\log X\big).
\end{equation}
Using Proposition \ref{mainprop}(ii) and the discreteness of the spectrum, we get
\begin{align}\notag
    \begin{split}
        S(f^{+})&=\sum_{|t_j|\ge 1}2d_t(f^{+})\hat{u}_j^2+\sum_{|t_j|<1}2d_t(f^{+})\hat{u}_j^2\\
        &= \sum_{|t_j|\ge 1}2d_t(f^{+})\hat{u}_j^2+O\big(X^{1/2}\big).
    \end{split}
\end{align}
Since $d_t(f^{+})$ is an even function of $t$, see e.g. \eqref{df}, after using dyadic decomposition we get the
bound
\begin{align}\label{C-S}
    \begin{split}
        \sum_{|t_j|\ge 1}2d_t(f^{+})\hat{u}_j^2  &\ll\sum_{n=0}^{\infty}\Bigg(\sum_{2^n\le t_j<2^{n+1}}2\big|d_t(f^{+})\big|\hat{u}_j^2\Bigg)\\
        &\ll \sum_{n=0}^{\infty}\sup_{\substack{2^n\le t_j<2^{n+1}}}\big|d_{t_j}(f^{+})\big|\Bigg( \sum_{2^n\le t_j<2^{n+1}}\hat{u}_j^2\Bigg).
    \end{split}
\end{align}
From Proposition \ref{mainprop} and Lemma \ref{periods}, we compute
\begin{align}
    \begin{split}
        S(f^{+})\ll X^{1/2}\sum_{n=0}^{\infty} 2^{-n/2}\min\Big\{2^{n},XY^{-1}\Big\}+X^{1/2}.
    \end{split}
\end{align}
We split the sum according to $n<\log_2(X/Y)$ and $n>\log_2(X/Y)$. We get
\begin{align}\label{Sf++}
    \begin{split}
        S(f^{+})&\ll X^{1/2}\sum_{n<\log_2(X/Y)}2^{-n/2}\min\Big\{2^{n},XY^{-1}\Big\} \\
        &\quad+ X^{1/2}\sum_{n\ge \log_2(X/Y)}2^{-n/2}\min\Big\{2^{n},XY^{-1}\Big\}+X^{1/2} \\
        &\ll 
       .
    \end{split}
\end{align}
With this result for the spectral side and the analysis for the geometric side of our trace formula in Theorem \ref{traceformula}, we can finally show Theorem \ref{mainthm}.
From the above we have that 
\begin{align}
    \begin{split}
        \sum_{\substack{{\gamma \in H_1 \backslash \Gamma -H_1 /H_1} \\ B(\gamma)<1 }}\widetilde{g}^{+}\big(B(\gamma)^2\big)+\sum_{\substack{{\gamma \in H_1 \backslash \Gamma -H_1 /H_1} \\ B(\gamma)>1 }}g^{+}\big(B(\gamma)^2\big)=\sum_{1/2<s_j\le1}\frac{2}{\pi}\gamma_1(s_j)\hat{u}_j^2X^{s_j}\\
        \quad+O\Big(XY^{-1/2}+Y+X^{1/2}\log X\Big). \;
    \end{split}
\end{align}
But the sum for the exceptional terms $\gamma$, i.e. $B(\gamma)<1$, is $O(1)$ from Lemma \ref{Exc}. Hence we have that
$$\sum_{\substack{{\gamma \in H_1 \backslash \Gamma -H_1 /H_1} \\ B(\gamma)>1 }}g^{+}\big(B(\gamma)^2\big)=\sum_{1/2<s_j\le1}\frac{2}{\pi}\gamma_1(s_j)\hat{u}_j^2X^{s_j}+O\Big(XY^{-1/2}+Y+X^{1/2}\log X\Big).$$
We balance the error term by choosing $Y=X^{2/3}$ producing a bound of order $O\big(X^{2/3}\big).$ It follows that
$$\sum_{\substack{{\gamma \in H_1 \backslash \Gamma -H_1 /H_1} \\ B(\gamma)>1 }}g^{+}\big(B(\gamma)^2\big)=\sum_{1/2<s_j\le1}\frac{2}{\pi}\gamma_1(s_j)\hat{u}_j^2X^{s_j}+O\big(X^{2/3}\big).$$
We work similarly for the sums of $g^{-}$, $\widetilde{g}^{-}$ and use equation \eqref{Nineq} to get that
$$N(X,l)=\sum_{\substack{{\gamma \in H_1 \backslash \Gamma /H_1} \\ B(\gamma) \leq 1 }}1+\sum_{1/2<s_j\le1}\frac{2}{\pi}\gamma_1(s_j)\hat{u}_j^2X^{s_j}+O\big(X^{2/3}\big).$$
Since the first sum is of order $O(1)$ we conclude that
$$N(X,l)=\sum_{1/2<s_j\le1}\frac{2}{\pi}\gamma_1(s_j)\hat{u}_j^2X^{s_j}+O\big(X^{2/3}\big).$$
\end{proof}

\begin{remark}
Let $\nu$ be such that $H_1=\langle P^\nu \rangle$ for $P$ primitive. For $s=1$, using the fact that 
$$\hat{u_0}= \frac{\mu}{\sqrt{\hbox{vol}(\Gamma \backslash \mathbb{H})}\nu},$$
the contribution to the main term of $N(X,l)$ is $$2\frac{\gamma_1(1)}{\pi}\hat{u_0}^2X=\frac{2}{\pi} \frac{\mu^2}{\hbox{vol}(\Gamma \backslash \mathbb{H})\nu^2}X.$$
\end{remark}
\begin{remark}\label{Rem}
Let $l_1,l_2$ be representatives of two closed geodesics and $H_1$ a hyperbolic subgroup of $\Gamma$ that corresponds to $l_1$ and $H_2$ a hyperbolic subgroup that corresponds to $l_2$. We assume that $l_1$ and $\tau^{-1} \cdot l_2$ lie on the imaginary axis $I$, for some $\tau \in \hbox{PSL}_2(\mathbb{R})$. 
Then we consider the series
$$\widetilde{K}(z):= \sum_{\gamma \in H_2 \backslash \Gamma} k(v(\tau^{-1}\gamma z)).$$ 
Similarly to the proof of Theorem \ref{traceformula} we have
\begin{align}
    \begin{split}
        \int_{l_1} \widetilde{K}(z)ds &= \sum_{\gamma \in H_2\backslash \Gamma} \sum_{\gamma_0 \in H_1} \int_{l_1}k(v(\tau^{-1}\gamma \gamma_0 z))ds \\
        &=\sum_{\gamma \in H_2\backslash \Gamma /H_1}  \int_Ik(v(\tau^{-1}\gamma z))ds.
    \end{split}
\end{align}
The geometric side of Theorem \ref{traceformula} is the same for $B(\tau^{-1}\gamma)$ in place of $B(\gamma)$.
The difference in the spectral side case comes from the periods $\hat{u}_j$. The sum appearing there becomes
$$\sum_{t_j}2d_{t_j}(f)\int_{l_1}u_j(z)ds \int_{l_2}u_j(z)ds.$$
For the sum $\sum_{t_j\le T}\int_{l_1}u_j(z)ds\int_{l_2}u_j(z)ds$, we use Lemma \ref{periods} and the Cauchy--Schwarz inequality to obtain the same bound as the one in the proof of Theorem \ref{mainthm}. We have
$$\Bigg|\sum_{0\le t_j \le T}\int_{l_1}u_j(z)ds\int_{l_2}u_j(z)ds\Bigg| \ll T,$$
which can be used to show equations \eqref{C-S} and \eqref{Sf++} in this case and the same error term as in Theorem \ref{mainthm} follows for $N(X,l_1,l_2)$.
\end{remark}

\section{Second moment of the error term}\label{meansquare}
In this section we prove Theorem \ref{SecondMomthm} for averaging the second moment of the error term in counting $N(X,l)$.\\
We recall that the main term of our hyperbolic lattice counting problem is
$$M(X,l)= \sum_{1/2< s_j \le 1} \frac{2}{\pi}\gamma_1(s_j)\hat{u}_j^2X^{s_j}, $$ and the error term is 
$$ E(X,l)=N(X,l)-M(X,l).$$
Also we define the error term related to a test function $f$ as
$$E_f(X,l)=f(1)\hbox{len}(l)+\sum_{\gamma \in H_1\backslash \Gamma-H_1 / H_1}Q(B(\gamma))-2\sum_{1/2 < s_j \le 1 }d_{t_j}(f)\hat{u}_j^2,$$
where
\[ Q(B(\gamma))= 
\begin{cases}
q(B(\gamma)), & B(\gamma)>1, \\
\widetilde{q}(B(\gamma)), & B(\gamma)<1,
\end{cases}
\] 
and $q$, $\widetilde{q}$ are given by \eqref{qz} and \eqref{qtilde}.
In the proof of Theorem \ref{mainthm} (see \eqref{Esf},\eqref{Sf++}) we showed that 
\begin{align}
    \begin{split} \label{E_f}
        E_{f^{+}}(X,l) &=O(XY^{-1/2}+X^{1/2}), \\
        E_{f^{-}}(X,l) &=O(XY^{-1/2}+X^{1/2}) , 
    \end{split}
\end{align}
$$E_{f^{-}}(X,l) <E(X,l)+O\Big(Y+X^{1/2}\log X\Big)<E_{f^{+}}(X,l).$$
We choose $Y$ such that $X^{1/2}\log X\ll Y \ll X$, then 
$$ E_{f^{+}}(X,l)<E(X,l)+O(Y)<E_{f^{-}}(X,l).$$
We can now use Theorem \ref{Sieve} to prove Proposition \ref{SecondMomProp}:
\begin{proof}
We follow the proof of \cite[Prop. 5.3]{CP}.
In the following we write $f$ for either $f^{-}$ or $f^{+}$ defined in \eqref{f+} and \eqref{f-}. For the average sum of $E(X,l)$ over points $X_1,X_2, \dots ,X_R$ we have 
$$ \sum_{m=1}^R|E(X_m,l)|^2 \ll \sum_{m=1}^R|E_f(X_m,l)|^2+RY^2.$$
We define the sum
$$S(X,T)= 2\sum_{T<|t_j|\le 2T}d_{t_j}(f)\hat{u}_j^2.$$ 
and split the $t_j$ in the following intervals:
\begin{itemize}
  \item[] $A_1=\{t_j : 0<|t_j|\le 1\}$,
  \item[] $A_2=\{t_j : 1\le |t_j| \le X^2Y^{-2}\}$,
  \item[] $A_3=\{t_j : |t_j|>X^2Y^{-2}\}.$
\end{itemize}
Let also 
$$S_i=2\sum_{t_j\in A_i}d_{t_j}(f)\hat{u}_j^2,$$
hence
$$ E_f(X,l)=S_1+S_2+S_3.$$
We start by bounding $S_1$: using the estimates for $d_{t_j}(f)$ from Proposition \ref{mainprop} we have
$$\sum_{t_j\in A_1}2d_{t_j}(f)\hat{u}_j^2 \ll X^{1/2}\sum_{|t_j|<1}t_j^{-3/2}\min \{t_j,X/Y\}\hat{u}_j^2 \ll X^{1/2} \ll Y,$$
since there exist finitely many spectral parameters $|t_j|\le 1$.\\
For $S_3$ we compute
\begin{align}
    \begin{split}
        \sum_{t_j\in A_3}2d_{t_j}(f)\hat{u}_j^2 &\ll \sum_{|t_j|>X^2Y^{-2}}|t_j|^{-3/2}\min\{t_j,X/Y\}X^{1/2}\hat{u}_j^2 \\
        &\ll \sum_{t_j>X^2Y^{-2}}t_j^{-3/2}X^{3/2}Y^{-1}\hat{u}_j^2.
    \end{split}
\end{align}
By partial summation and Lemma \ref{periods} we get 
$$S_3 \ll X^{1/2} \ll Y.$$
Hence we have shown that
$$E_f(X,l) = \sum_{t_j\in A_2}2d_{t_j}(f)\hat{u}_j^2+O(Y).$$
We use dyadic decomposition. We consider values $T=2^k, k=0,1, \dots ,\log_2(X^2Y^{-2})$ and we sum over only those $T's$ to get
$$ E_f(X,l) \ll \sum_{1\le T=2^k \le X^2Y^{-2}}S(X,T)+Y$$ and we also add for $X_1,X_2,\dots X_R$ to get
$$ \sum_{m=1}^R|E_f(X_m,l)|^2 \ll \sum_{m=1}^R \Big|\sum_{1\le T=2^k \le X^2Y^{-2}}S(X_m,T)\Big|^2+RY^2.$$
We use the Cauchy--Schwarz inequality for the inner sum
$$\Big|\sum_{1\le T=2^k \le X^2Y^{-2}}S(X_m,T)\Big|^2 \ll \log X \sum_{1\le T=2^k \le X^2Y^{-2}}|S(X_m,T)|^2.$$
Combining the last two inequalities we show
$$\sum_{m=1}^R|E_f(X_m,l)|^2 \ll \log X \sum_{1 \le T=2^k \le X^2Y^{-2}}\Big(\sum_{m=1}^R|S(X_m,T)|^2\Big)+RY^2.$$
In Proposition \ref{mainprop} we have written the Huber transform of f as
$$d_{t_j}(f)= X^{1/2}(a(t,X/Y)X^{it}+b(t,X/Y)X^{-it})+O(t^{-3/2}),$$
where $a(t,X/Y)$ and $b(t,X/Y)$ are functions satisfying 
$$a(t,X/Y),b(t,X/Y) \ll |t|^{-3/2}\min\{|t|,X/Y\}.$$
Now using Theorem \ref{Sieve} for $a_j x_{\nu}^{i t_j}=d_{t_j}(f)\hat{u}_j$, we have
$$\sum_{m=1}^R\Big|\sum_{T<|t_j|\le2T}d_{t_j}(f)\hat{u}_j^2\Big|^2 \ll (T+X\delta^{-1})||a||_*^2$$ 
or, equivalently,
$$\sum_{m=1}^R|S(X_m,T)|^2 \ll (T+X\delta^{-1})||a||_*^2 \; .$$
Let us now bound the norm $||a||_*^2$:
\begin{align}
    \begin{split}
        ||a||_*^2 &\ll \sum_{T<|t_j|\le 2T}\big||t_j|^{-3/2}\min\{t_j,X/Y\}X^{1/2}\hat{u}_j\big|^2 \\
        & \ll XT^{-3}\min\{T^2,X^2Y^{-2}\}\sum_{T<|t_j|\le2T}|\hat{u}_j|^2.
    \end{split}
\end{align}
By Lemma \ref{periods} we can see that 
$$||a||_*^2\ll XT^{-2}\min\{T^2,X^2Y^{-2}\},$$
hence overall we have shown:
$$\sum_{m=1}^R|S(X_m,T)|^2 \ll (T+X\delta^{-1})XT^{-2}\min\{T^2,X^2Y^{-2}\}.$$
Using the last bound we have
\begin{align}
    \begin{split}
        \sum_{m=1}^R|E(X_m,l)|^2 &\ll \log X \sum_{1\le T=2^k\le X^2Y^{-2}}\Bigg(\sum_{m=1}^{R}|S(X_m,T)|^2\Bigg)+RY^2\\
        &\ll \log X \sum_{1\le T=2^k\le X^2Y^{-2}}(T+X\delta^{-1})XT^{-2}\min\{T^2,X^2Y^{-2}\}+RY^2.
    \end{split}
\end{align}
We get the bound
\begin{align}
    \begin{split}
        \sum_{m=1}^R|E(X_m,l)|^2 &\ll X\log X\sum_{1\le T=2^k <XY^{-1}}T+X^2\delta^{-1}\log X\sum_{1\le T=2^k < XY^{-1}}1\\
        &\quad +X^3Y^{-2}\log X\sum_{XY^{-1}\le T=2^k < X^2Y^{-2}}T^{-1} \\
        &\quad + X^4\delta^{-1}Y^{-2}\log X\sum_{XY^{-1}\le T=2^k < X^2Y^{-2}}T^{-2}+RY^2.
    \end{split}
\end{align}
By trivial bounds for each term individually we can show that
$$\sum_{m=1}^R|E(X_m,l)|^2\ll X^2Y^{-1}\log X+\delta^{-1}X^2\log^2 X+RY^2.$$
We notice that the optimal choice for $Y$ is $Y=R^{-1/3}X^{2/3}$ and that choice gives us the bound
$$\sum_{m=1}^R|E(X_m,l)|^2\ll R^{1/3}X^{4/3}\log X+\delta^{-1}X^2\log^2 X,$$
which concludes the proof of Proposition \ref{SecondMomProp}.
\end{proof}
The proof of Theorem \ref{SecondMomthm} also follows from the proof of \cite[Theorem 5.4]{CP}:
\begin{proof}
To prove equation \eqref{1/R} we choose $\delta^{-1} \ll RX^{-1}$ and $R>X^{1/2}$ in the bound 
$$\sum_{m=1}^R|E(X_m,l)|^2\ll R^{1/3}X^{4/3}\log X+\delta^{-1}X^2\log^2 X.$$
We get
$$\sum_{m=1}^R|E(X_m,l)|^2\ll R^{1/3}X^{4/3}\log X+RX\log^2 X\ll RX\log^2 X.$$
To prove equation \eqref{1/X} we choose the points $X_i$ to be equally spaced in $[X,2X]$ with $\delta=R^{-1}X$. The function $|E(x,l)|^2$ is integrable, because it has finitely many discontinuities as a function of $x$. Taking the mesh $X/R$ to tend to $0$, we have 
$$\sum_{m=1}^R|E(X_m,l)|^2\frac{X}{R} \to \int_X^{2X}|E(x,l)|^2dx.$$
This implies
$$\frac{1}{X}\int_X^{2X}|E(x,l)|^2dx \ll X\log^2 X.$$
\end{proof}

\appendix
\section{Special functions}

 We list the main properties of special functions that we use. For reference see \cite{gradshteyn2007}.

By \cite[eq. 8.328.1]{gradshteyn2007} we have the Stirling's approximation for the Gamma function:
\begin{equation}\label{Stir}
    \lim_{|y|\to \infty}|\Gamma(x+iy)| e^{\frac{\pi}{2}|y|}|y|^{\frac{1}{2}-x}\sim \sqrt{2\pi}.
\end{equation}

\begin{defn}
For $|z|<1$ the Gauss hypergeometric function is defined by the power series 
$${}_2F_1(a,b,c;z)=1+\frac{a\cdot b}{c \cdot 1}z+\frac{a(a+1)b(b+1)}{c(c+1)\cdot 1 \cdot 2}z^2+\frac{a(a+1)(a+2)b(b+1)(b+2)}{c(c+1)(c+2)\cdot 1 \cdot 2 \cdot 3}z^3+ \dots.$$
\end{defn}
It follows that ${}_2F_1(a,b,c;0)=1$.
For $\Re b>\Re c >0$ the Gauss hypergeometric function has the integral representation \cite[eq. 9.111]{gradshteyn2007}
$${}_2F_1(a,b,c;z)=\frac{1}{B(b,c-b)}\int_0^1t^{b-1}(1-t)^{c-b-1}(1-tz)^{-a}dt.$$
The derivative of ${}_2F_1(a,b,c;z)$ with respect to $z$ is given by
\begin{equation}\label{der}
\frac{d}{dz}\big({}_2F_1(a,b,c;z)\big)=\frac{ab}{c}{}_2F_1(a+1,b+1,c+1;z).\end{equation}
We also need the transformation formula \cite[eq. 9.132.2]{gradshteyn2007}
\begin{align}
\begin{split}
    {}_2F_1(a,b,c;z)&=\frac{\Gamma(c)\Gamma(b-a)}{\Gamma(b)\Gamma(c-a)}(-z)^{-a}{}_2F_1\Big(a,a+1-c,a+1-b;\frac{1}{z}\Big) \\
    &\quad+ \frac{\Gamma(c)\Gamma(b-a)}{\Gamma(a)\Gamma(c-b)}(-z)^{-b}{}_2F_1\Big(b,b+1-c,b+1-a;\frac{1}{z}\Big),
\end{split}
\end{align}
where $\arg z<\pi$,\: $a-b\neq \pm m,\quad m=0,1,2, \dots$
and formulas \cite[eq. 9.136.1--9.136.2]{gradshteyn2007}:
Let $$A=\frac{\Gamma\big(a+b+\frac{1}{2}\big)\sqrt{\pi}}{\Gamma\big(a+\frac{1}{2}\big)\Gamma\big(b+\frac{1}{2}\big)}, \quad B=\frac{-\Gamma\big(a+b+\frac{1}{2}\big)2\sqrt{\pi}}{\Gamma(a)\Gamma(b)}\;;$$ then we have
$${}_2F_1\Bigg(2a,2b,a+b+\frac{1}{2};\frac{1-\sqrt{z}}{2}\Bigg)=A\; {}_2F_1\big(a,b,\frac{1}{2};z\big)+B\sqrt{z}\; {}_2F_1\Bigg(a+\frac{1}{2},b+\frac{1}{2},\frac{3}{2};z\Bigg)$$
and
$${}_2F_1\Bigg(2a,2b,a+b+\frac{1}{2};\frac{1+\sqrt{z}}{2}\Bigg)=A\; {}_2F_1\big(a,b,\frac{1}{2};z\big)-B\sqrt{z}\; {}_2F_1\Bigg(a+\frac{1}{2},b+\frac{1}{2},\frac{3}{2};z\Bigg).$$
The Pochhammer symbol is defined as $(a)_k=\Gamma(a+k)/\Gamma(a)$. The series $${}_pF_q (\alpha_1, \alpha_2, \dots \alpha_p; \beta_1, \beta_2, \dots , \beta_q; z)=\sum_{k=0}^{\infty}\frac{(\alpha_1)_k(\alpha_2)_k \cdots (\alpha_p)_k}{(\beta_1)_k(\beta_2)_k \cdots (\beta_p)_k}\frac{z^k}{k!} $$
is called a generalized hypergeometric series.\\
For the series ${}_pF_q$ we use the integral formula \cite[eq. 7.512.12]{gradshteyn2007}:
\begin{align}
    \begin{split}
      \int_0^1x^{\mu-1}(1-x)^{\nu-\gamma-1}&{}_pF_q(a_1, \dots ,a_p,b_1, \dots ,b_q;ax)dx\\
      &=\frac{\Gamma(\mu)\Gamma(\nu)}{\Gamma(\mu+\nu)}\: _{p+1}F_{q+1}(\nu,a_1, \dots ,a_p,\mu+\nu,b_1, \dots ,b_q;a),
    \end{split}
\end{align}
where $\Re \mu>0,\; \Re \nu>0,\; p\le q+1$,\; and, if $p=q+1$ then $|a|<1$.
For $z\in \mathbb{R}$ with $|\arg({-z})|< \pi$ we also write the transformation formula \cite[(16.8.8)]{NIST}:
$$_{q+1}F_q(a_1, \dots ,a_{q+1},b_1, \dots ,b_q;z)=\sum_{j=1}^{q+1}\Bigg(\prod_{\substack{k=1 \\ k \neq j}}^{q+1}\frac{\Gamma(a_k-a_j)}{\Gamma(a_k)}\Bigg/ \prod_{k=1}^q\frac{\Gamma(b_k-a_j)}{\Gamma(b_k)}\Bigg)\widetilde{w}_j(z) \; ,$$
with $$\widetilde{w}_j(z)=(-z)^{a_j}\:_{q+1}F_q\Big(a_j,1-b_1+a_j, \dots ,1-b_q+a_j, 1-a_1+a_j,\dots, *, \dots , 1-a_{q+1}+a_j;\frac{1}{z}\Big),$$
for $j=1, \dots , q+1$, where $*$ indicates that the entry $1-a_j+a_j$ is omitted.
We also need the following transformation formulas for the hypergeometric function from \cite[pp. 1008--1010, 9.131.3, 9.133.3, 9.137.12]{gradshteyn2007}:
\begin{equation}\label{9.131}
    _2F_{1}(a,b,c;z)=(1-z)^{c-a-b}\:_2F_1(c-a,c-b,c,z),  
\end{equation}
\begin{equation}\label{9.133}
     _2F_{1}(a,b,2b;z)=\:_2F_1(a,b,a+b+1/2,4z(1-z)),  
\end{equation}
for $|z|\le1/2,|z(1-z)|\le1/4$, and 
\begin{equation}\label{9.137}
    c\;_2F_{1}(a,b,c;z)-c\;_2F_1(a+1,b,c,z)+bz\;_2F_1(a+1,b+1,c+1,z)=0.
\end{equation}

\begin{defn}
The associate Legendre function of the first kind is defined as (see \cite[eq. 8.704]{gradshteyn2007})
\begin{equation}\label{Leg}
P_{\nu}^{\mu}(z)=\frac{1}{\Gamma(1-\mu)}\Bigg(\frac{z+1}{z-1}\Bigg)^{\mu/2}{}_2F_1\Bigg(-\nu,\nu+1,1-\mu;\frac{1-z}{2}\Bigg).
\end{equation}
\end{defn}
We use the integral representation for $P_{\nu}^{-\mu}(z)$ \cite[eq. 8.713.3]{gradshteyn2007}
$$P_{\nu}^{-\mu}(z)=\sqrt{\frac{2}{\pi}}\frac{\Gamma\big(\mu+\frac{1}{2}\big)\big(z^2-1)^{\mu/2}}{\Gamma(\nu+\mu+1)\Gamma(\mu-\nu)}\int_{0}^{\infty} \frac{\cosh\big(\nu+\frac{1}{2}\big)t}{(z+\cosh t)^{\mu+\frac{1}{2}}}dt,$$
where $\Re z>-1$,\; $|\arg(z\pm 1)|<\pi$, $\Re(\nu+\mu)>-1$ and $\Re(\mu-\nu)>0$.

\nocite{*}
\printbibliography 
\end{document}